\documentclass[10pt]{article}
\usepackage{epsfig}
\usepackage{amsmath}
\usepackage{amsfonts}
\usepackage{amsthm}
\usepackage{subfigure}
\usepackage{verbatim}
\usepackage{psfrag}
\newtheorem{thm}{Theorem}[section]

\newtheorem{lemma}[thm]{Lemma}

\newtheorem{cor}[thm]{Corollary}
\newtheorem{remark}{Remark}

\newcommand{\bpm}{\begin{pmatrix}}
\newcommand{\abs}[1]{\left\vert#1 \right\vert}
\def\epm{\end{pmatrix}}
\def\bs{\boldsymbol}
\def\d{d}
\def\eq{\begin{equation}}
\def\endeq{\end{equation}}
\def\eqarr{\begin{eqnarray}}
\def\endeqarr{\end{eqnarray}}
\def\eqnn{\begin{equation*}}
\def\endeqnn{\end{equation*}}
\def\ds{\begin{displaystyle}}
\def\endds{\end{displaystyle}}

\def\e{\epsilon}
\def\Ai{\text{Ai}}
\def\cauchy{\mathfrak{C}}

\newcommand{\Airy}{\mathbb{A}}

\usepackage{color}
\definecolor{light-blue}{rgb}{0.8,0.85,1}
\definecolor{blue}{rgb}{0,0,1}
\definecolor{red}{rgb}{1,0,0}

\begin{document}

\title{{\bf Asymptotics of Tracy-Widom distributions and
the total integral of a Painlev\'e II function}}

\author{{\bf Jinho Baik}
\footnote{Department of Mathematics, University of Michigan, Ann
Arbor, MI, 48109, baik@umich.edu} \footnote{Courant Institute of
Mathematical Sciences, New York University}, {\bf Robert
Buckingham}\footnote{Department of Mathematics, University of
Michigan, Ann Arbor, MI, 48109, robbiejb@umich.edu} and {\bf Jeffery
DiFranco}\footnote{Department of Mathematics, University of
Michigan, Ann Arbor, MI, 48109, jeffcd@umich.edu}}

\date{\today}
\maketitle

\begin{abstract}
The Tracy-Widom distribution functions involve integrals of a
Painlev\'e II function starting from positive infinity. In this
paper, we express the Tracy-Widom distribution functions in terms of
integrals starting from minus infinity. There are two
consequences of these new representations. The first is the
evaluation of the total integral of the Hastings-McLeod solution of
the Painlev\'e II equation. The second is the evaluation of the
constant term of the asymptotic expansions of the Tracy-Widom
distribution functions as the distribution parameter approaches minus
infinity. For the GUE
Tracy-Widom distribution function, this gives an alternative proof
of the recent work of Deift, Its, and Krasovsky. The constant terms
for the GOE and GSE Tracy-Widom distribution functions are new.
\end{abstract}

\section{Introduction}

\label{intro}

Let $F_1(x), F_2(x)$, and $F_4(x)$ denote the GOE, GUE, and GSE
Tracy-Widom distribution functions, respectively. They are defined
as  \cite{Tracy:1994, Tracy:1996}
\begin{equation}\label{eq:TW0}
    F_1(x)= F(x)E(x), \qquad F_2(x)= F(x)^2, \qquad F_4(x)=
    \frac12\bigg\{ E(x)+ \frac1{E(x)} \bigg\}F(x),
\end{equation}
where
\begin{equation}\label{eq:TW1}
    F(x)= \exp\bigg( -\frac12 \int_x^\infty R(s)ds\bigg),
    \qquad
    E(x)= \exp\bigg( -\frac12 \int_x^\infty q(s)ds \bigg).
\end{equation}
Here the (real) function $q(x)$ is the solution to the Painlev\'e II
equation
\begin{equation}\label{eq:PII}
    q''=2q^3+xq,
\end{equation}
that satisfies the boundary condition
\begin{equation}\label{eq:PIIbdry}
  q(x)\sim \Ai(x),
  \qquad x\to +\infty.
\end{equation}
Recall \cite{Abramowitz:1965-book} that the Airy function $\Ai(x)$
satisfies $\Ai''(x)=x\Ai(x)$ and
\begin{equation}
    \Ai(x) \sim \frac{1}{2\sqrt{\pi}x^{1/4}}e^{-\frac23x^{3/2}},
    \qquad x\to +\infty.
\end{equation}
There is a unique global solution $q(x)$ to the
equation~\eqref{eq:PII} with the condition~\eqref{eq:PIIbdry} (the
Hastings-McLeod solution) \cite{Hastings:1980}. The function
$R(s)$ is defined as
\begin{equation}\label{eq:Rdef}
   R(x) =  \int_x^\infty (q(s))^2ds.
\end{equation}
By taking derivatives and using~\eqref{eq:PII}
and~\eqref{eq:PIIbdry} (see, for example, (1.15) of
\cite{Tracy:1994} and (2.6) of \cite{Baik:2000}), the function
$R(x)$ can also be written as
\begin{equation}\label{eq:Rqlocal}
    R(x)= (q'(x))^2-x(q(x))^2 - (q(x))^4.
\end{equation}
Integrating by parts, $F(x)$ can be written as
\begin{equation}
    F(x)= \exp\bigg( -\frac12 \int_x^\infty (s-x)(q(s))^2ds\bigg),
\end{equation}
which is commonly used in the literature.

Notice that~\eqref{eq:TW1} involves integrals from $x$ to positive
infinity. The main results of this paper are the following
representations of $F(x)$ and $E(x)$, which involve integrals
from minus infinity to $x$.
\begin{thm}\label{thm1}
For $x<0$,
\begin{equation}\label{Fdown}
    F(x)= 2^{1/48}e^{\frac12 \zeta'(-1)}
    \frac{e^{-\frac{1}{24}|x|^3}}{|x|^{1/16}} \exp\bigg\{ \frac12\int_{-\infty}^x\bigg(
    R(y)-\frac14y^2+\frac1{8y}\bigg)dy \bigg\},
\end{equation}
where $\zeta(z)$ is the Riemann-zeta function, and
\begin{equation}
\label{Edown}
    E(x) = \frac1{2^{1/4}} e^{- \frac{1}{3\sqrt{2}} |x|^{3/2}}
    \exp\bigg\{ \frac12 \int_{-\infty}^x \bigg(
    q(y)-\sqrt{\frac{|y|}2} \bigg) dy \bigg\}.
\end{equation}
\end{thm}

\begin{remark}
The formula \eqref{Fdown} also follows from the recent work
\cite{Deift:2006} of Deift, Its, and Krasovsky. See
subsection~\ref{sec:asycons} below for further discussion.
\end{remark}

The integrals in \eqref{Fdown} and \eqref{Edown} converge. Indeed,
it is known that \cite{Hastings:1980, Deift:1995}
\begin{equation}\label{eq:qasymp}
    q(x) = \sqrt{\frac{-x}{2}} \bigg( 1+
    \frac1{8x^3}-\frac{73}{128 x^6} + \frac{10219}{1024x^9}+O(|x|^{-12}) \bigg), \qquad x\to
    -\infty.
\end{equation}
This asymptotic behavior of $q$ was obtained using the integrable
structure of the Painlev\'e II equation (see, for example,
\cite{Fokas:2006-book}). The coefficients of the higher terms in
the above asymptotic expansion can also be computed recursively
(see for example, Theorem 1.28 of \cite{Deift:1995}). For
$R(x)$,~\eqref{eq:Rqlocal} and~\eqref{eq:qasymp} imply that
\begin{equation}\label{eq:PIIneg}
    R(x) = \frac{x^2}{4} \bigg( 1- \frac{1}{2x^3}+ \frac{9}{16x^6}
    - \frac{128}{62 x^9} + O(x^{-12})\bigg), \qquad x\to -\infty.
\end{equation}

We now discuss two consequences of Theorem \ref{thm1}.

\subsection{Total integrals of $q(x)$ and $R(x)$}

Comparing with~\eqref{eq:TW1}, Theorem~\ref{thm1} is equivalent to
the following.
\begin{cor}\label{cor:total}
For $c<0$,
\begin{equation}
    \int_c^\infty R(y)dy + \int_{-\infty}^c \bigg(
R(y)-\frac14y^2+\frac1{8y}\bigg)dy
    = - \frac1{24}\log 2 -
    \zeta'(-1) + \frac1{12}|c|^3 + \frac18\log|c|
\end{equation}
and
\begin{equation}
    \int_c^\infty q(y)dy + \int_{-\infty}^c \bigg(
    q(y)-\sqrt{\frac{|y|}2} \bigg) dy
    = \frac12 \log 2 +\frac{\sqrt{2}}{3}|c|^{3/2} .
\end{equation}
\end{cor}

These formulas should be compared with the evaluation of the total
integral of the Airy function \cite{Abramowitz:1965-book}:
\begin{equation}
    \int_{-\infty}^\infty \Ai(y)dy =1.
\end{equation}
Recall that the Airy differential equation is the small amplitude
limit of the Painlev\'e II equation. Unlike the Airy function,
$R(x)$ and $q(x)$ do not decay as $x\to-\infty$, and hence we need
to subtract out the diverging terms in order to make the integrals
finite.

\subsection{Asymptotics of Tracy-Widom distribution functions as $x\to
-\infty$}\label{sec:asycons}

Using formulas~\eqref{eq:TW1} and~\eqref{eq:PIIneg}, Tracy and
Widom computed that (see Section 1.D of \cite{Tracy:1994}) as $x\to
-\infty$,
\begin{equation}\label{eq:F2as}
    F_2(x) = \tau_2
     \frac{e^{-\frac1{12}|x|^3}}{|x|^{1/8}} \bigg(
    1+ \frac{3}{2^6|x|^3} + O(|x|^{-6}) \bigg)\\
\end{equation}
for some undetermined constant $\tau_2$. The constant $\tau_2$ was
conjectured in the same paper \cite{Tracy:1994} to be
\begin{equation}\label{eq:constant}
    \tau_2=2^{1/24}e^{\zeta'(-1)}.
\end{equation}
This conjecture~\eqref{eq:constant} was recently proved by Deift,
Its, and Krasovsky \cite{Deift:2006}. In this paper, we present an
alternative proof of~\eqref{eq:constant}. Moreover, we also compute
the similar constants $\tau_1$ and $\tau_4$ for the GOE and GSE
Tracy-Widom distribution functions. The asymptotics similar
to~\eqref{eq:F2as} follow from~\eqref{eq:TW0}
and~\eqref{eq:qasymp}: as $x\to -\infty$,
\begin{eqnarray}\label{eq:F1as}
    F_1(x) &=& \tau_1
\frac{e^{-\frac1{24}|x|^3-\frac1{3\sqrt{2}}|x|^{3/2}}}{|x|^{1/16}}
    \bigg( 1 -\frac1{24\sqrt{2}|x|^{3/2}}+ O(|x|^{-3})\bigg), \\
    \label{eq:F4as}
    F_4(x) &=& \tau_4
    \frac{e^{-\frac1{24}|x|^3+\frac1{3\sqrt{2}}|x|^{3/2}}}{|x|^{1/16}}
    \bigg( 1 +\frac1{24\sqrt{2}|x|^{3/2}}+ O(|x|^{-3})\bigg).
\end{eqnarray}
Using~\eqref{eq:qasymp} and~\eqref{eq:PIIneg}, Theorem~\ref{thm1}
implies the following.
\begin{cor}
As $x\to -\infty$,
\begin{equation}
    F(x) = 2^{1/48}e^{\frac12\zeta'(-1)}
\frac{e^{-\frac1{24}|x|^3}}{|x|^{1/16}} \bigg(
    1+ \frac{3}{2^7|x|^3} +    O(|x|^{-6})\bigg)
\end{equation}
and
\begin{equation}
    E(x) = \frac{1}{2^{1/4}}  e^{-\frac1{3\sqrt{2}}|x|^{3/2}}
    \bigg( 1 -\frac1{24\sqrt{2}|x|^{3/2}}+ O(|x|^{-3})\bigg).
\end{equation}
Hence
\begin{equation}
    \tau_1= 2^{-11/48}e^{\frac12 \zeta'(-1)},
    \qquad \tau_2= 2^{1/24}e^{\zeta'(-1)},
    \qquad \tau_4= 2^{-35/48} e^{\frac12 \zeta'(-1)}.
\end{equation}
\end{cor}

Conversely, using~\eqref{eq:qasymp} and~\eqref{eq:PIIneg}, this
Corollary together with~\eqref{eq:TW1} implies
Corollary~\ref{cor:total}, and hence Theorem~\ref{thm1}.

This is one example of so-called \emph{constant problems} in random
matrix theory. One can ask the same question of evaluating the
constant term in the asymptotic expansion in other distribution
functions such as the limiting gap distribution in the bulk or in
the hard edge. For the gap probability distribution in the bulk
scaling limit which is given by the Fredholm determinant of the
sine-kernel, Dyson \cite{Dyson:1976} first conjectured the constant
term for $\beta=2$ in terms of $\zeta'(-1)$ using a formula in an
earlier work \cite{Widom:1971} of Widom. This conjecture was proved
by Ehrhardt \cite{Ehrhardt:2006a} and Krasovsky
\cite{Krasovsky:2004}, independently and simultaneously. A third proof
was given in \cite{Deift:2006b}. The constant problem for $\beta=1$
and $\beta=4$ ensembles in the bulk scaling limit was recently
obtained by Ehrhardt \cite{Ehrhardt:2006b}. For the hard edge of the
$\beta$-Laguerre ensemble associated with the weight $x^me^{-x}$,
the constant was obtained by Forrester \cite{Forrester:1994}
(equation (2.26a)) when $m$ is a non-negative integer and $2/\beta$
is a positive integer.

The above limiting distribution functions in random matrix theory
are expressed in terms of a Fredholm determinant or an integral
involving a Painlev\'e function. For example, the proof of
\cite{Deift:2006} used the Freldhom determinant formula of the GUE
Tracy-Widom distribution:
\begin{equation}\label{eq:FredAiry}
    F_2(x)=\det(1-\Airy_x),
\end{equation}
where $\Airy_x$ is the operator on $L^2((x,\infty))$ whose kernel is
\begin{equation}
    \Airy(u,v)= \frac{\Ai(u)\Ai'(v)-\Ai'(u)\Ai(v)}{u-v}.
\end{equation}
In terms of the Fredholm determinant formula, the difficulty comes
from the fact that even if we know all the eigenvalues
$\lambda_j(x)$ of $\Airy_x$, we still need to evaluate the product
$\prod_{j=1}^\infty(1-\lambda_j(x))$. When one uses the Painlev\'e
function, one faces a similar difficulty of evaluating the total
integral of the Painlev\'e function.

We remark that the asymptotics as $x\to +\infty$ of $F(x)$ and
$E(x)$ (and hence $F_\beta(x)$) are, using~\eqref{eq:PIIbdry},
\begin{eqnarray}
  F(x) &=& 1-\frac{e^{-\frac43x^{3/2}}}{32\pi x^{3/2}} \bigg( 1- \frac{35}{24 x^{3/2}}
  + O(x^{-3})\bigg), \\
  E(x) &=& 1 - \frac{e^{-\frac23x^{3/2}}}{4\sqrt{\pi} x^{3/2}} \bigg( 1- \frac{41}{48 x^{3/2}}
  + O(x^{-3})\bigg).
\end{eqnarray}

\subsection{Outline of the proof}

The Tracy-Widom distribution functions are the limits of a variety
of objects such as the largest eigenvalue of certain ensembles of
random matrices, the length of the longest increasing subsequence of
a random permutation, the last passage time of a certain last
passage percolation model, and the height of a certain random growth
model (see, for example, the survey \cite{Majumdar:2007}).  Dyson
\cite{Dyson:1976} exploited this notion of universality to solve the
constant problem for the sine-kernel determinant.  Namely, among the
many different quantities whose limit is the sine-kernel
determinant, he chose one for which the associated constant term is
explicitly computable (specifically, a certain Toeplitz determinant
on an arc for which the constant term had been obtained by Widom
\cite{Widom:1971}), and then took the appropriate limit while
checking the limit of the constant term. However, the rigorous
proof of this idea was only obtained in the subsequent work
of Ehrhardt \cite{Ehrhardt:2006a} and Krasovsky
\cite{Krasovsky:2004}. In order to apply this idea for $F_2(x)$, the
key step is to choose the appropriate approximate ensemble. In the
work of Deift, Its, and Krasovsky \cite{Deift:2006}, the authors
started with the Laguerre unitary ensemble and took the appropriate
limit while controlling the error terms. In this paper, we use the
fact that $F_\beta(x)$ is a (double-scaling) limit of a
Toeplitz/Hankel determinant.

Let $D_n(t)$ denote the $n\times n$ Toeplitz determinant with symbol
$f(e^{i\theta})=e^{2t\cos (\theta)}$ on the unit circle: \eq
\begin{split}
\label{Dn}
D_n(t) & = \; \det\left(\frac{1}{2\pi}\int_{-\pi}^{\pi}e^{2t\cos\theta}e^{i(j-k)\theta}\d\theta\right)_{0\leq j,k\leq n-1}\\
   & = \; \frac{1}{(2\pi)^n n!}\int_{[-\pi,\pi]^n}e^{2t\sum_{j=1}^n \cos\theta_j}\prod_{1\leq k<\ell\leq n}\abs{e^{i\theta_k}-e^{i\theta_\ell}}^2 \prod_{j=1}^n\d\theta_j.
\end{split}
\endeq
Note that some references (e.g. \cite{Deift:1998-book}) define $D_n$
as an $(n+1)\times (n+1)$ determinant, whereas others (e.g.
\cite{Baik:2001b}) use our convention.  In studying the asymptotics
of the length of the longest increasing subsequence in random
permutations, in \cite{Baik:1999}, the authors proved that when
\begin{equation}
   n=[2t+xt^{1/3}],
\end{equation}
as $t\to \infty$,
\begin{equation}\label{eq:DntoF2}
    e^{-t^2}D_n(t) \to F_2(x).
\end{equation}
The idea of the proof of~\eqref{eq:DntoF2} in \cite{Baik:1999} is as
follows. The Toeplitz determinants are intimately related to
orthogonal polynomials on the unit circle. Let $p_j(z)=\kappa_j
z^j+\cdots$ be the orthonormal polynomial of degree $j$ with respect to the
weight $e^{2 t \cos\theta} \frac{d\theta}{2\pi}$:
\begin{equation}\label{orthcondition}
\int_{-\pi}^{\pi}p_j(e^{i\theta})\overline{p_k(e^{i\theta})}
e^{2 t \cos\theta}\frac{d\theta}{2\pi}=\delta_{jk}
\qquad \text{ for } j,k\geq 0.
\end{equation}
If $\kappa_j>0$ then $p_j$ is unique.  We denote by $\pi_j(z;t)=\pi_j(z)$
the monic orthogonal
polynomial: $p_k(z)=\kappa_k \pi_j(z)$. Then (see, for example,
\cite{Szego:1975-book}) the leading coefficient
$\kappa_j=\kappa_j(t)$ is given by
\begin{equation}\label{kappan}
\kappa_j(t)= \sqrt{\frac{D_j(t)}{D_{j+1}(t)}}.
\end{equation}
As the strong Szeg\"o limit theorem implies that $D_n(t)\to
e^{t^2}$ as $n\to\infty$ for fixed $t$, the left-hand-side
of~\eqref{eq:DntoF2} can be written as
\begin{equation}\label{eq:produp}
   e^{-t^2}D_n(t) = \prod_{q=n}^\infty
\frac{D_q(t)}{D_{q+1}(t)} = \prod_{q=n}^\infty
\kappa_q^2(t).
\end{equation}
The basic result of \cite{Baik:1999} is that
\begin{equation}\label{eq:PIIcov}
    \kappa_q^2(t) \sim 1- \frac{R(y)}{t^{1/3}},
    \qquad  t\to \infty, \quad
    q=[2 t +yt^{1/3}]
\end{equation}
for $y$ in a compact subset of $\mathbb{R}$. (In \cite{Baik:1999},
the notations $v(x)=-R(x)$ and $u(x)=-q(x)$ are used.) Hence
formally, as $ t\to\infty$ with $n=2 t +xt^{1/3}$,
\begin{equation}
    \log\big( e^{-t^2}D_n(t) \big)
    = \sum_{q=n}^\infty \log \big( \kappa_q^2(t) \big)
    \sim t^{1/3} \int_x^\infty
\log\bigg(1-\frac{R(y)}{t^{1/3}}\bigg)
    dy
    \sim -\int_x^\infty R(y)dy = \log F_2(x).
\end{equation}

The first step of this paper is to write, instead
of~\eqref{eq:produp}, \eq\label{eq:proddown} e^{-t^2}D_{n}(t) =
e^{-t^2}\prod_{q=1}^{n}\frac{D_q(t)}{D_{q-1}(t)} = e^{-t^2}
\prod_{q=1}^{n} \frac1{\kappa_{q-1}^{2}(t)}.
\endeq
Here $D_0(t):=1$. Then formally, we expect that as $ t\to\infty$
with $n=2 t +xt^{1/3}$,~\eqref{eq:proddown} converges to an integral
from $-\infty$ to $x$. For this to work, we need the asymptotics of
$\kappa_q(t)$ for the whole range of $q$ and $t$ such that $1\le
q\le 2 t +xt^{1/3}$ as $ t\to\infty$.

It turns out it is more convenient to write,
for an arbitrary fixed $L$, \eq\label{eq:proddown2}
e^{-t^2}D_{n}(t) =
e^{-t^2}D_L(t)\prod_{q=L+1}^{n}\frac{D_q(t)}{D_{q-1}(t)
} = e^{-t^2} D_L(t)\prod_{q=L+1}^{n}
\frac1{\kappa_{q-1}^{2}(t)}.
\endeq
We introduce another fixed large number $M>0$ and write \eq
\label{sum-parts} \log(e^{-t^2}D_{n}) = -t^2 +
\underbrace{\log(D_L)}_\text{\vspace{.2in}exact part} +
\underbrace{\sum_{q=L+1}^{[2 t -Mt^{1/3}-1]}
\log(\kappa_{q-1}^{-2})}_\text{Airy part} + \underbrace{\sum_{q= [2
t -Mt^{1/3}]}^{[2 t +xt^{1/3}]}
\log(\kappa_{q-1}^{-2})}_\text{Painlev\'e part}.
\endeq
Since $L$ and $M$ are arbitrary, we can compute the desired limit by
computing
\begin{equation}
    \lim_{L, M\to\infty} \lim_{ t\to\infty} \log
    (e^{-t^2}D_n(t)), \qquad
    n=[2 t +xt^{1/3}].
\end{equation}
From~\eqref{eq:PIIcov}, the Painlev\'e part converges to a finite
integral of $R(y)$ from $y=-M$ to $y=x$ as $t\to\infty$.
For the Airy part, we need the asymptotics of $\kappa_q(t)$ for
$L+1\le q\le 2 t -Mt^{1/3}$ as $ t\to\infty$ for fixed $L, M>0$. The
paper \cite{Baik:1999} obtains a weak one-sided bound of
$\kappa_q(t)$ for $\epsilon t\le q\le 2 t -Mt^{1/3}$ as $
t\to\infty$, where $\epsilon>0$ is small but fixed. The technical
part of this paper is to compute the leading asymptotics of
$\kappa_q(t)$ in $L+1\le q\le 2 t -Mt^{1/3}$ with proper control of
the errors so that the Airy part converges. The advantage of
introducing $L$ is that we do not need small values of $q$, which
simplifies the analysis. The calculation is carried out in Section
\ref{airy-part}. Finally, for the exact part, the asymptotics of
$D_L(t)$ as $ t\to\infty$ are straightforward using a
steepest-descent method since the size of the determinant is fixed
and only the weight varies. The limit is given in terms of the
Selberg integral for the $L\times L$ Gaussian unitary ensemble, which is
given by a product of Gamma functions, the Barnes G-function. The
asymptotics of the Barnes G-function as $L\to\infty$ are related to
the term $\zeta'(-1)$ (see~\eqref{G-at-large-z} below). The
computation is carried out in Section~\ref{sec-exact-part}.

Now we outline the proof of the formula~\eqref{Edown} for $E(x)$. In
the study of symmetrized random permutations it was proven in
\cite{Baik:2001a, Baik:2001b} that, in a similar double scaling
limit, certain other determinants converge to $F_1(x)$ and $F_4(x)$.
But it was observed in \cite{Baik:2001a, Baik:2001b} that these
determinants can be expressed in terms of $\kappa_q(t)$ and
$\pi_q(0;t)$ for the \emph{same} orthonormal
polynomials~\eqref{orthcondition} above. Hence by using the same
idea for $D_n(t)$, we only need to keep track of $\pi_q(0;t)$ in the
asymptotic analysis of the orthogonal polynomials. See
Section~\ref{E-of-x} below for more details.

This paper is organized as follows. In Section~\ref{sec-exact-part},
the asymptotics of the exact part of~\eqref{sum-parts} are computed.
We compute the asymptotics of the Airy part in Section~\ref{airy-part}.
The proof of~\eqref{Fdown} for $F(x)$ in Theorem~\ref{thm1} is
then given in Section~\ref{sec:Fproof}. The proof of~\eqref{Edown}
for $E(x)$ in Theorem~\ref{thm1} is given in Section~\ref{E-of-x}.

While we were writing up this paper, Alexander Its told us that
there is another way to compute the constant term for $E(x)$ using a
formula in \cite{Baik:2001b}. This idea will be explored in a later
publication together with Its to compute the total integrals of
other Painlev\'e solutions, such as the Ablowitz-Segur solution.

\bigskip

\noindent {\bf Acknowledgments.} The authors would like to thank P.
Deift and A. Its for useful communications. The work of the first
author was supported in part by NSF Grant \# DMS-0457335 and the
Sloan Fellowship.  The second and third authors were partially supported
by NSF Focused Research Group grant \# DMS-0354373.

\section{The exact part}
\label{sec-exact-part}

We compute the exact part of~\eqref{sum-parts}. From
equation~\eqref{Dn}, \eq D_L(t) = \frac{1}{(2\pi)^L L!}\int_{[-\pi,
\pi]^L}e^{2 t \sum_{j=1}^L \cos\theta_j}\prod_{1\leq k<\ell\leq
L}\abs{e^{i\theta_k}-e^{i\theta_\ell}}^2 \prod_{j=1}^L\d\theta_j.
\endeq
Following the standard stationary phase method of restricting each
integral to a small interval $-\e\leq\theta\leq\e$ and expanding
$e^{i\theta}$ and $e^{2 t \cos\theta}$ in Taylor series,  $D_L(t)$
is approximately \eq \frac{1}{(2\pi)^L L!}\int_{[-\e,\e]^L} e^{2 t
L- t \sum_{j=1}^L\theta_j^2}\prod_{1\leq k<\ell\leq L}\left\vert
\theta_k - \theta_\ell \right\vert^2 \prod_{j=1}^L\d\theta_j
\endeq
as $t\to\infty$. By extending the range of integration to
$\mathbb{R}^L$, we obtain \eq \lim_{t\to\infty}D_L(t)\cdot
\left(\frac{e^{2 t L}}{(2\pi)^L}D_L^\text{Herm}(t)\right)^{-1} =
1,
\endeq
where
\eq
\begin{split}
 D_L^\text{Herm}(t) & = \;
\frac{1}{L!}\int_{[-\infty,\infty]^L} e^{- t
\sum_{j=1}^L\theta_j^2}\prod_{1\leq k<\ell\leq n}\left\vert
\theta_k - \theta_\ell \right\vert^2 \prod_{j=1}^L\d\theta_j.
\end{split}
\endeq
This integral is known as a Selberg integral and is computed
explicitly as (see for example, \cite{Mehta:1991-book}, equation
(17.6.7)) \eq\label{DLHerm}
  D_L^\text{Herm}(t) = \; \frac{\pi^{L/2}}{2^{L(L-1)/2} t
  ^{L^2/2}}\prod_{q=0}^{L-1}q!
  = \; \frac{\pi^{L/2}}{2^{L(L-1)/2} t
  ^{L^2/2}} G(L+1),
\endeq
where $G(z)$ denotes the Barnes $G$-function, or double gamma
function.  Some
properties of the Barnes $G$-functions are (see, for example,
\cite{Voros:1987, Choi:2003}) \eq G(z+1)=\Gamma(z)G(z) \endeq \eq
G(1)=G(2)=G(3)=1.
\endeq
\eq \label{G-of-one-half} \log G\left(\frac{1}{2}\right) =
\frac{1}{24}\log 2 - \frac{1}{4}\log\pi + \frac{3}{2}\zeta'(-1),
\quad \log G\left(\frac{3}{2}\right) = \frac{1}{24}\log 2 +
\frac{1}{4}\log\pi + \frac{3}{2}\zeta'(-1),
\endeq
\begin{equation}
  G(n)=1!2!\cdots (n-2)!, \qquad n=2,3,4,\cdots.
\end{equation}
\eq \label{G-at-large-z} \log G(z+1) = \frac{z^2}{2}\log z
-\frac{3}{4}z^2 + \frac{z}{2}\log(2\pi)-\frac{1}{12}\log z +
\zeta'(-1) + O\left(\frac{1}{z^2}\right) \quad \text{ as }
z\rightarrow\infty.
\endeq
Therefore, \eq\label{exact-part-result} \lim_{L\to\infty}
\lim_{t\to\infty}\left(\log(D_L)-\bigg\{2L t -\frac{L^2}{2}\log(2t)
+\left(\frac{L^2}2-\frac1{12}\right)\log L -\frac34L^2
+\zeta'(-1)\bigg\}\right) = 0.
\endeq

\section{The Airy part}
\label{airy-part}

The main result of this section is Lemma~\ref{lem:Airy} which
computes \eq \label{double-scaling}
\lim_{L,M\rightarrow\infty}\lim_{ t\to\infty}\sum_{q=L+1}^{[2 t
-Mt^{1/3}-1]} \log(\kappa_{q-1}^{-2}),
\endeq
the Airy part of \eqref{sum-parts}.  We use the notation
\begin{equation}
    \gamma=\frac{2t}{q}.
\end{equation}

It is well known that the leading coefficients $\kappa_q^{-2}$ of
orthonormal polynomials can be expressed in terms of the solution of
a matrix Riemann-Hilbert Problem (RHP) \cite{FIK}. We start with the
RHP for $m^{(5)}$ defined in Section 6 (p.1156) of \cite{Baik:1999},
which is obtained through a series of explicit transformations of
the original RHP for orthogonal polynomials. For notational ease,
we drop the tildes Let $\theta_c$ be defined such that
$0<\theta_c<\pi$ and $\sin^2\frac{\theta_c}{2} = \frac{1}{\gamma}$.
For $q$ in the regime $L+1\le q\le 2t-Mt^{1/3}-1$, we have
$\gamma>1$. Define the contours $C_1 =
\{e^{i\theta}:\theta_c<|\theta|\leq\pi\}$ and $C_2 =
\{e^{i\theta}:0\leq|\theta|\leq\theta_c\}$ with the orientations
given as in Figure~\ref{contour-Sigma5}. Also define the contours
$C_{\text{in}}$ and $C_{\text{out}}$ as in
Figure~\ref{contour-Sigma5}. Let $\Sigma^{(5)} = C_1 \cup C_2 \cup
C_\text{in} \cup C_\text{out}$. Now let $m^{(5)}(z)=m^{(5)}(z;t,q)$
be the solution to the following RHP: \eq
\begin{cases} m^{(5)} \text{ is analytic in }
z\in\mathbb{C}\backslash\Sigma_5 \\
m^{(5)}_+(z) = m^{(5)}_-(z)v^{(5)}(z) \text{ for } z\in\Sigma_5 \\
m^{(5)} = I \text{ as } z\rightarrow\infty  \end{cases}
\endeq
where the jump matrix $v^{(5)}(z)=v^{(5)}(z;t,q)$ is given by
 \begin{equation}\label{v5}
 v^{(5)}(z)=\begin{cases}\left(\begin{matrix} 0 & 1 \\ -1 & 0\end{matrix}\right) & z\in C_2\\
 \left(\begin{matrix} 1 & e^{-2q\alpha}\\ 0 & 1\end{matrix}\right) &z\in C_1\\
 \left(\begin{matrix} 1 & 0\\ e^{2q\alpha} & 1\end{matrix}\right) & z\in C_\text{in}\cup
 C_\text{out}.
 \end{cases}
 \end{equation}
Here \eq \label{alpha} \alpha(z) =
-\frac{\gamma}{4}\int_\xi^z\frac{s+1}{s^2}\sqrt{(s-\xi)(s-\overline{\xi})}\d
s,
\endeq
where $\xi=e^{i\theta_c}$ and the branch is chosen to be analytic in
$\mathbb{C}\backslash\overline{C_2}$ and
$\sqrt{(s-\xi)(s-\overline{\xi})}\sim +s$ for
$s\rightarrow\infty$.
\begin{figure}[h]
\centering
\psfrag{O}[1][1][1][1]{\begin{large}$\mathcal{O}_{\xi}$\end{large}}
\psfrag{U}[1][1][1][1]{\begin{large}$\mathcal{O}_{\overline{\xi}}$\end{large}}
\psfrag{C1}{$C_1$}
\psfrag{C2}{$C_{2}$}
\psfrag{Cin}{$C_{\text{in}}$}
\psfrag{Cot}{$C_{\text{out}}$}
\psfrag{xi}{$\xi$}
\psfrag{xi2}{$\overline{\xi}$}
\psfrag{c}{$C_1^+$}
\psfrag{d}{$C_1^-$}
\psfrag{m}{$C_{\text{in}}^+$}
\psfrag{n}{$C_{\text{in}}^-$}
\psfrag{o}{$C_{\text{out}}^+$}
\psfrag{l}{$C_{\text{out}}^-$}
\mbox{\subfigure[The contour $\Sigma^{(5)}$ for $m^{(5)}$.]
{\label{contour-Sigma5}\epsfig{figure=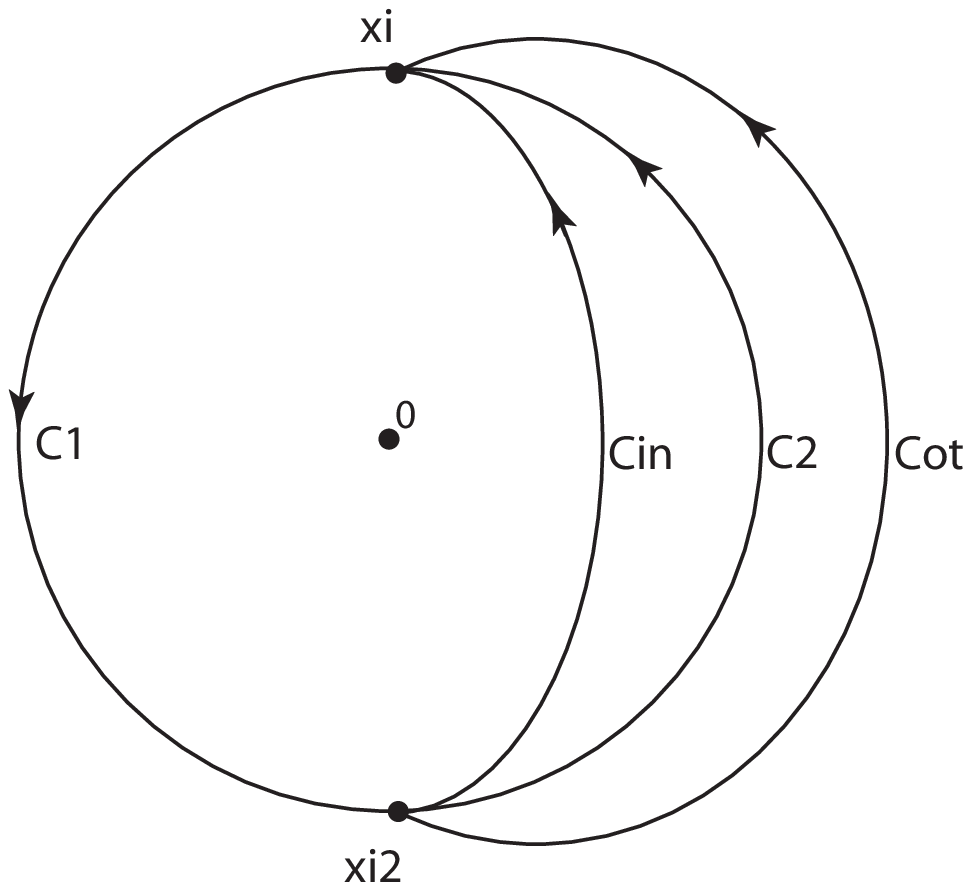, width=2in}} \subfigure[Introduction of $\mathcal{O}_\xi$ and $\mathcal{O}_{\overline{\xi}}.$]
{\epsfig{figure=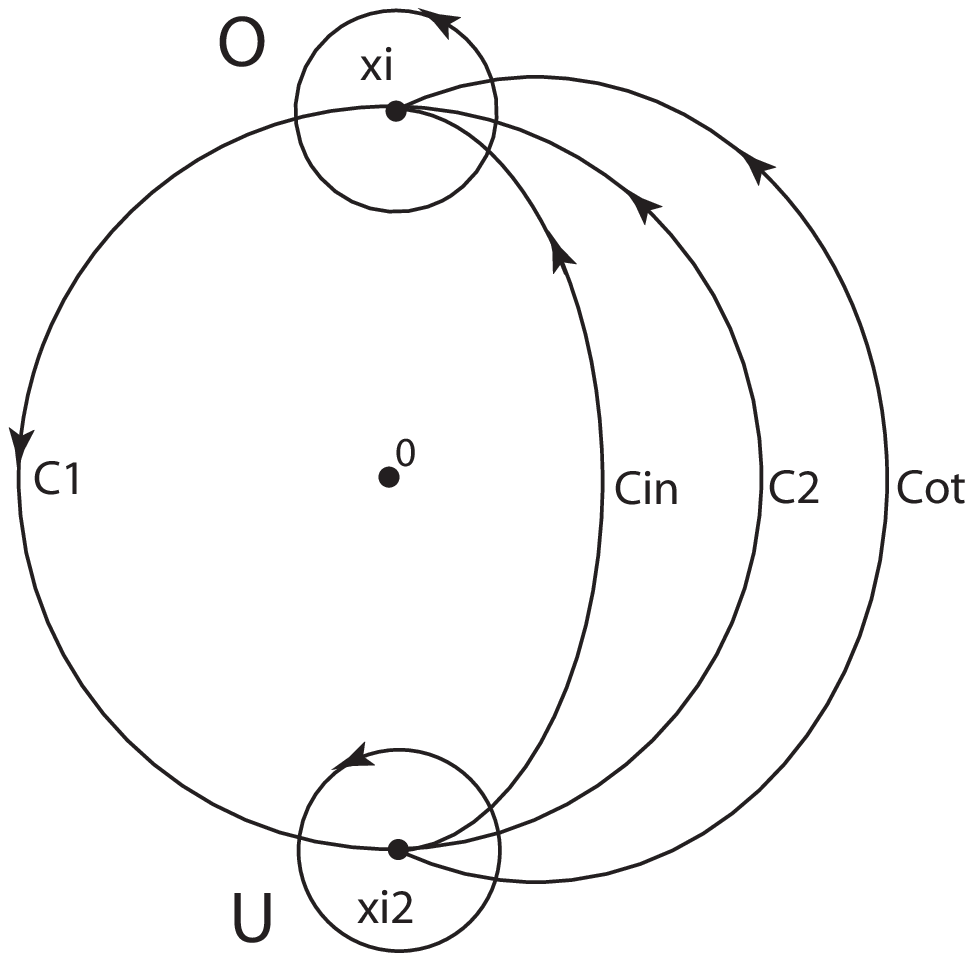,width=2in}} \subfigure[The contour $\Sigma$ for $m^{(R)}$.]{\label{contour-Sigma}\epsfig{figure=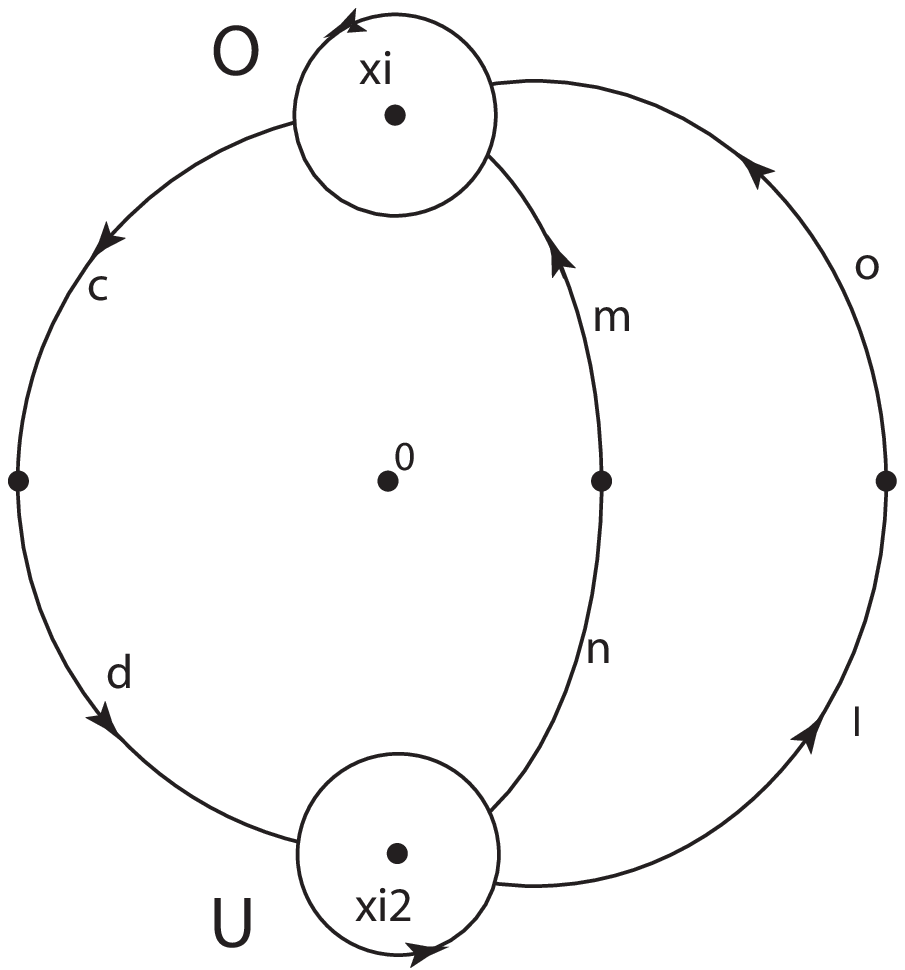,width=2in}}}
\caption{Contours used in the definition of the RHP for $m^{(5)}$ and $m^{(R)}$.}
\end{figure}
Then (see (6.40) of \cite{Baik:1999}) \eq \kappa_{q-1}^2 =
-e^{q(-\gamma+\log\gamma+1)}m_{21}^{(5)}(0)
\endeq
and
\begin{equation}
    \pi_q(0)=(-1)^qm^{(5)}_{11}(0).
\end{equation}

We analyze the solution $m^{(5)}$ to this RHP for the regime $L+1\le
q\le 2t-Mt^{1/3}-1$ as $t\to\infty$. Our analysis builds on the work
of \cite{Baik:1999} and makes two main technical improvements.  The
first is that the paper
\cite{Baik:1999} only considered the regime when $\epsilon t\le q$.
Hence $q$ necessarily grows to infinity. In this work, we allow $q$
to be finite. The second is that we compute a higher order
correction explicitly to the asymptotics obtained in
\cite{Baik:1999}. This higher-order correction contributes to the
sum~\eqref{double-scaling}.  In \cite{Baik:1999}, only a one-sided
bound of a similar sum was obtained.  We merely outline
the analysis for the parts that overlap with the analysis of
\cite{Baik:1999}.

From the construction of $\alpha$ in \cite{Baik:1999} we have that
$\abs{ e^{-\alpha(z)}}<1$ for $z\in C_1$ and $\abs{
e^{\alpha(z)}}<1$ for $z\in C_\text{in}\cup C_\text{out}$.  If we
formally take the limit of our jump matrix $v^{(5)}$ as
$q\rightarrow\infty$ the jumps on the contours $C_\text{in}$ and
$C_\text{out}$ approach the identity matrix and the jumps on $C_1$
and $C_2$ approach constant jumps.  This limiting RHP is solved
explicitly by
\begin{equation}
m^{(5, \infty)}(z)=\begin{pmatrix}
\frac{1}{2}\left(\beta+\beta^{-1}\right)
&\frac{1}{2i}\left(\beta-\beta^{-1}\right)
\\-\frac{1}{2i}\left(\beta-\beta^{-1}\right)
&\frac{1}{2}\left(\beta+\beta^{-1}\right) \end{pmatrix},
\end{equation}
where $\beta(z)=\left(\frac{z-\xi}{z-\xi^{-1}}\right)^{1/4}$, which
is analytic for $z\in\mathbb{C}\setminus\overline{C_2}$ and
$\beta\to 1$ as $z\rightarrow\infty$.  Note that
\begin{equation}
   m^{(5,\infty)}_{21}(0)=-\frac{1}{\sqrt{\gamma}}= -\frac{q}{2t},
   \qquad
   m^{(5,\infty)}_{11}(0)=\sqrt{\frac{\gamma-1}{\gamma}}=\sqrt{\frac{2t-q}{2t}}.
\end{equation}

However, the convergence of the jump matrix $v^{(5)}(z)$ is not
uniform near the points $\xi$ and $\overline{\xi}$, since
$\alpha(\xi)=\alpha(\overline{\xi})=0$.  Therefore a parametrix is
introduced around these points.  For fixed $\delta<\frac{1}{100}$,
define \eq
\mathcal{O}_{\xi}=\{z:|z-\xi|\leq\delta|\xi-\overline{\xi}|\}, \quad
\mathcal{O}_{\overline{\xi}}=\{z:|z-\overline{\xi}|\leq\delta|\xi-\overline{\xi}|\}.
\endeq
Note that the diameter of $\mathcal{O}_{\xi}$ is of order
$\frac{\sqrt{\gamma-1}}{\gamma}$ and varies as $t$ and $q$ vary. The
diameter approaches $0$ as $\gamma\to 1$ or $\gamma\to\infty$,
which happens when $q$ is close to $2t-Mt^{1/3}$ or $L+1$,
respectively. However, the point is that in the regime $L+1\le q\le
2t-Mt^{1/3}-1$, the diameter of $\mathcal{O}_{\xi}$ cannot shrink
``too fast.''  Therefore, the usual Airy parametrix for a domain
of fixed size still yields a good parametrix for the RHP in the regime
under consideration. The case
when $\gamma\to 1$ ``slowly'' was analyzed in \cite{Baik:1999} for the
leading asymptotics of $m^{(5)}$. In this section, we also analyze
the case when $\gamma\to \infty$ ``slowly,'' and also improve the work
in \cite{Baik:1999} to obtain a higher-order correction term.

Orient the boundary of both $\mathcal{O}_{\xi}$ and
$\mathcal{O}_{\overline{\xi}}$ in the counterclockwise direction.
Now as in \cite{Baik:1999} (see also \cite{Deift:1995}) for $z\in
\mathcal{O}_{\xi}\setminus\Sigma^{(5)}$ define the matrix-valued
function $m_p$ as
\begin{equation}\label{mp1}
m_p(z)=\begin{pmatrix}1 & -1 \\ -i & -i \end{pmatrix}\sqrt{\pi}
e^{i\pi/6}q^{\frac{\sigma_3}{6}}\left(\left(\frac{3}{2}\alpha(z)\right)^{\frac{2}{3}}
\left(\frac{z-\overline{\xi}}{z-\xi}\right)\right)^{\frac{\sigma_3}{4}}
\Psi\left(\left(\frac{3}{2}q\alpha(z)\right)^{2/3}\right)e^{q\alpha(z)\sigma_3},
\end{equation}
where $\omega=e^{2\pi i/3}$ and
\begin{equation}\label{eq:Airypara}
\Psi(s)=\begin{cases} \begin{pmatrix} \Ai(s) & \Ai(\omega^2 s)\\\Ai'(s) & \omega^2 \Ai'(\omega^2s)\end{pmatrix}e^{-\frac{\pi i}{6}\sigma_3}, &0 <\arg(s)<\frac{2\pi}{3},\\
\begin{pmatrix} \Ai(s) & \Ai(\omega^2 s)\\\Ai'(s) & \omega^2\Ai'(\omega^2s)\end{pmatrix}e^{-\frac{\pi i}{6}\sigma_3}\begin{pmatrix} 1 & 0 \\ -1 & 1\end{pmatrix}, & \frac{2\pi}{3}<\arg(s)<\pi,\\
\begin{pmatrix} \Ai(s) & -\omega^2 \Ai(\omega s)\\\Ai'(s) & -\Ai'(\omega s)\end{pmatrix}e^{-\frac{\pi i}{6}\sigma_3}\begin{pmatrix}1 & 0 \\ 1 & 1\end{pmatrix}, & \pi<\arg(s)<\frac{4\pi}{3},\\
\begin{pmatrix} \Ai(s) & -\omega^2\Ai(\omega s)\\\Ai'(s) & -\Ai'(\omega s)\end{pmatrix}e^{-\frac{\pi i}{6}\sigma_3}, &
\frac{4\pi}{3}<\arg(s)<2\pi.\end{cases}
 \end{equation}
We can define $m_p(z)=\overline{m_p(\overline{z})}$ for $z\in
\mathcal{O}_{\overline{\xi}}$ by (\ref{mp1}). For $z\notin
\Sigma^{(5)}\cup(\overline{\mathcal{O}_{\xi}\cup\mathcal{O}_{\overline{\xi}}})$,
let $m_p(z)=m^{(5,\infty)}(z)$. It is shown in \cite{Baik:1999} that
$m_p$ then solves a RHP that has the same jump conditions as
$m^{(5)}$ on the contour $C_2$ as well as on
$\Sigma^{(5)}\cap\mathcal{O}$, where we define
$\mathcal{O}=\mathcal{O}_{\xi}\cup\mathcal{O}_{\overline{\xi}}$.

Define  $R(z)= m^{(5)}(z)m_p^{-1}(z)$. Then $R$ solves a RHP on
$\Sigma =
\partial\mathcal{O} \cup ((C_1 \cup C_\text{in}\cup
C_\text{out})\cap \mathcal{O}^c)$ with jump
$v_{R}=m_{p-}v^{(5)}v^{-1}_pm^{-1}_{p-}$.  Explicitly, the jump
matrix $v_R$ is given by
\begin{equation} \label{VR}
v_{R}=I+ \begin{cases} 0, & z\in (\Sigma^{(5)}\cap \mathcal{O})\cup C_2,\\
m^{(5,\infty)}\begin{pmatrix} 0 & e^{-2 q\alpha}\\ 0 & 0 \end{pmatrix} (m^{(5,\infty)})^{-1}, & z\in C_1\cap\mathcal{O}^c,\\
m^{(5,\infty)}\begin{pmatrix} 0 & 0\\e^{2 q\alpha}& 0\end{pmatrix} (m^{(5,\infty)})^{-1}, & z\in (C_\text{in}\cup C_\text{out})\cap\mathcal{O}^c,\\
\frac{1}{q\alpha}v^{q\alpha}_{R}+v_R^E,
& z\in \partial\mathcal{O},
\end{cases}
\end{equation}
where $v^{q\alpha}_R$ is given explicitly in
Lemma~\ref{expansion-of-vR} below, and the matrix $v_R^E$ is defined
as $v_R^E= v_R-I-\frac1{q\alpha} v_R^{q\alpha}$ for
$z\in\partial\mathcal{O}$. Since $m^{(5)}(0) = R(0)m_p(0) =
R(0)m^{(5,\infty)}(0)$, we have $m_{21}^{(5)}(0) =
-\frac{1}{\sqrt{\gamma}}(R_{22}(0)-\sqrt{\gamma-1}R_{21}(0))$ and
$m_{11}^{(5)}(0) =
-\frac{1}{\sqrt{\gamma}}(R_{12}(0)-\sqrt{\gamma-1}R_{11}(0))$.
Therefore
\begin{equation}\label{eq:kappaR}
    \kappa_{q-1}^2= \frac1{\sqrt{\gamma}}e^{q(-\gamma+\log
    \gamma+1)} \big(R_{22}(0)-\sqrt{\gamma-1}R_{21}(0)\big)
\end{equation}
and
\begin{equation}\label{eq:piR}
    \pi_q(0)
    =
    -\frac{(-1)^q}{\sqrt{\gamma}} \big(R_{12}(0)-\sqrt{\gamma-1}R_{11}(0)
    \big).
\end{equation}

In \cite{Baik:1999}, for $z\in
\partial \mathcal{O}$, the jump matrix $v_R$ is approximated by the
identity matrix $I$ and the terms
$\frac{1}{q\alpha}v^{q\alpha}_{R}+v_R^E$ are treated as an error
(for the case when $et\le q$). For our purpose, we need to compute
the contribution from the next order term
$\frac{1}{q\alpha}v^{q\alpha}_{R}$ explicitly.

It will be shown in the following subsections that for any
$\epsilon>0$, there are $L_0$ and $M_0$ such that for fixed $L\ge
L_0$ and $M\ge M_0$, there is $t_0=t_0(L, M)$ such that
$||v_R-I||_{L^\infty(\Sigma)}<\epsilon$ for all $t\ge t_0$ and
$L+1\le q\le 2t-Mt^{1/3}-1$. This was shown in \cite{Baik:1999} for
$\epsilon_1 t \le q\le 2t-Mt^{1/3}-1$. Then we proceed via the
standard Riemann-Hilbert analysis as, for example, in
\cite{Baik:1999}. Let $\cauchy(f)(z) := \frac{1}{2\pi
i}\int_\Sigma\frac{f(s)}{s-z}\d s$ be the Cauchy operator defined
for $z\notin\Sigma$. For $z\in\Sigma$, $\cauchy_-(f)(z)$ is defined
as the nontangential limit of $\cauchy(f)(z')$ as $z'$ approaches
$z$ from the right-hand side of $\Sigma$.  Define the operator
$\cauchy_R(f)=\cauchy_-(f(v_{R}-I))$ for $f\in L^2(\Sigma)$ and the
function $\mu = I + (1-\cauchy_R)^{-1}\cauchy_R I$.  A simple
scaling argument shows that $\cauchy_R$ is a uniformly bounded
operator for $L+1\le q\le 2t-Mt^{1/3}-1$. Since the supremum norm of
$v_R-I$ can be made as small as necessary, we find that for $L$ and
$M$ fixed but chosen large enough, $(1-\cauchy_R)^{-1}$ is a bounded
$L^2$ operator with norm uniformly bounded for $t$ sufficiently
large for all $q$ such that $L+1\le q\le 2t-Mt^{1/3}-1$. By the
theory of Riemann-Hilbert problems, \eq R(z) - I = \frac{1}{2\pi
i}\int_\Sigma\frac{\mu(s)(v_{R}-I)}{s-z}\d s.
\endeq

Define the contours
$\Sigma^{\pm}$ as the part of $\Sigma$ in the upper-half and lower-half planes,
respectively.  That is,
\eq
\Sigma^\pm = \Sigma \cap (\pm\Im(z)>0).
\endeq
Also define
\eq
C_1^\pm = C_1\cap\Sigma^\pm, \quad C_\text{in}^\pm = C_\text{in}\cap\Sigma^\pm, \quad C_\text{out}^\pm = C_\text{out}\cap\Sigma^\pm
\endeq
as shown in figure \ref{contour-Sigma}.
Now, by the Schwartz-reflexivity of $v_R$ (see \cite{Baik:1999}, p. 1159) and
$\mu$,
\eq
\frac{1}{2\pi i}\int_{\Sigma^-}\frac{\mu(s)(v_R(s)-I)}{s}\d s = \overline{\frac{1}{2\pi i}\int_{\Sigma^+}\frac{\mu(s)(v_R(s)-I)}{s}\d s}
\endeq
and therefore
\eq
R(0) - I = \Re\left[\frac{1}{\pi i}\int_{\Sigma^+}\frac{\mu(s)(v_{R}-I)}{s}\d s\right].
\endeq
We write this as
\begin{equation}\label{eq:R1to5}
   R(0)-I = R^{(1)}+R^{(2)}+R^{(3)}+R^{(4)}+R^{(5)}
\end{equation}
where
\eq \notag R^{(1)} = \Re\left[\frac{1}{\pi
i}\int_{\partial\mathcal{O}_\xi}\frac{1}{q\alpha}v_{R}^{q\alpha}(s)\frac{\d
s}{s}\right],
\endeq
\eq R^{(2)} = \Re\left[\frac{1}{\pi
i}\int_{\partial\mathcal{O}_\xi}\mu(s)\cdot v_R^E(s)\frac{\d
s}{s}\right], \quad R^{(3)} = \Re\left[\frac{1}{\pi
i}\int_{C_1^+}\mu(s)(v_{R}(s)-I)\frac{\d s}{s}\right],
\endeq
\eq \notag R^{(4)} = \Re\left[\frac{1}{\pi
i}\int_{C_{\text{in}}^+\cup
C_{\text{out}}^+}\mu(s)(v_{R}(s)-I)\frac{\d s}{s}\right], \quad
R^{(5)} = \Re\left[\frac{1}{\pi
i}\int_{\partial\mathcal{O}_\xi}(\mu(s)-I)\frac{1}{q\alpha}v_{R}^{q\alpha}(s)\frac{\d
s}{s}\right].
\endeq
Hence using~\eqref{eq:kappaR},
\begin{equation}\label{eq:73}
\begin{split}
    &\sum_{q=L+1}^{[2 t -Mt^{1/3}-1]}
    \log(\kappa_{q-1}^{-2})\\
    &=\sum_{q=L+1}^{[2 t -Mt^{1/3}-1]}
    \bigg\{ -q(-\gamma+\log \gamma +1) + \frac12 \log\gamma - \log
    \big(R_{22}(0)-\sqrt{\gamma-1}R_{21}(0)\big) \bigg\}\\
    &= \sum_{q=L+1}^{[2 t -Mt^{1/3}-1]}
    \bigg\{ -q\big(-\frac{2t}{q}+\log\left(\frac{2 t }{q}\right) +1\big) +
    \frac{1}{2}\log\left(\frac{2 t }{q}\right)\\
    &\qquad
    - \log \big(1+R^{(1)}_{22}(0)-\sqrt{\gamma-1}R^{(1)}_{21}(0)\big)
    - \log\bigg( 1+ \sum_{i=2}^5
    \frac{R^{(i)}_{22}(0)-\sqrt{\gamma-1}R^{(i)}_{21}(0)}
    {1+R^{(1)}_{22}(0)-\sqrt{\gamma-1}R^{(1)}_{21}(0)} \bigg)
    \bigg\}.
\end{split}
\end{equation}

\subsection{Calculation of
$\bs{R^{(1)}= \Re\left[\frac{1}{\pi i}
\int_{\partial\mathcal{O}_\xi}\frac{1}{q\alpha}v_{R}^{q\alpha}(s)\frac{\d
s}{s}\right]}$}

First, we compute $v_{R}^{q\alpha}$ explicitly.
\begin{lemma}
\label{expansion-of-vR} For $z\in \partial\mathcal{O}_{\xi}$, the
jump matrix $v_R(z)$ can be written as $v_{R} = I +
\frac{1}{q\alpha}v^{q\alpha}_{R} + v_R^E$ where
\begin{equation}
v^{q\alpha}_{R} = \frac{1}{2}\bpm d_1\beta^2+c_1\beta^{-2} &
i(d_1\beta^2-c_1\beta^{-2}) \\ i(d_1\beta^2-c_1\beta^{-2}) &
-(d_1\beta^2+c_1\beta^{-2})\epm, \qquad c_1 = \frac{5}{72}, \quad
d_1 = -\frac{7}{72}
\end{equation} and
\begin{equation}
    v_R^E=O\left(\frac{1}{|q\alpha|^2}\right).
\end{equation}
\end{lemma}
\begin{proof}
On $\partial\mathcal{O}_\xi$, $v_{R} = m_{p-}m_{p+}^{-1}$. On this
contour $m_{p-}=m^{(5,\infty)}$ and $m_{p+}$ is given by
\eqref{mp1}.  Thus for $z\in\partial \mathcal{O}_{\xi}$ \eq
v_R=m^{(5,\infty)}e^{-q\alpha(z)\sigma_3}
\Psi^{-1}\left(\left(\frac{3}{2}q\alpha(z)\right)^{2/3}\right)
\left(\left(\frac{3}{2}\alpha(z)\right)^{\frac{2}{3}}
\left(\frac{z-\overline{\xi}}{z-\xi}\right)\right)^{-\frac{\sigma_3}{4}}
e^{-i\pi/6}q^{-\frac{\sigma_3}{6}}
\frac{1}{\sqrt{\pi}}\begin{pmatrix}1 & -1 \\ -i & -i
\end{pmatrix}^{-1}\label{vR-boundaryO}.
\endeq

For $0<\arg(s)<\frac{2\pi}{3}$, consider (see~\eqref{eq:Airypara})
\eq
\label{Psi-definition} \Psi(s) = \bpm \Ai(s) & \Ai(\omega^2s) \\
\Ai'(s) & \omega^2\Ai'(\omega^2s)\epm e^{-(i\pi/6)\sigma_3}
\endeq
with $\omega=e^{2i\pi/3}$.  From Abramowitz and Stegun
\cite{Abramowitz:1965-book} (10.4.59) and (10.4.61), for
$|\arg(s)|<\pi$,
\eq
\label{Airy-asymptotics}
\Ai(s) = \frac{e^{-(2/3)s^{3/2}}}{2\sqrt{\pi}s^{1/4}}\left(1-\frac{c_1}{\frac{2}{3}s^{3/2}}+O\left(\frac{1}{|s|^3}\right)\right)
\endeq
\eq
\label{Airy-prime-asymptotics}
\Ai'(s) = -\frac{s^{1/4}e^{-(2/3)s^{3/2}}}{2\sqrt{\pi}}\left(1-\frac{d_1}{\frac{2}{3}s^{3/2}}+O\left(\frac{1}{|s|^3}\right)\right)
\endeq
wherein $c_1=\frac{5}{72}$ and $d_1=-\frac{7}{72}$.
Also note the identity
\eq
\Ai(s)+\omega\Ai(\omega s)+\omega^2\Ai(\omega^2 s) = 0
\endeq
and, from Abramowitz and Stegun (10.4.11.13),
\eq
W[\Ai(s),\Ai(\omega s)] = \frac{1}{2\pi}e^{-i\pi/6}
\endeq
\eq
W[\Ai(s),\Ai(\omega^2 s)] = \frac{1}{2\pi}e^{i\pi/6}
\endeq
\eq
W[\Ai(\omega s),\Ai(\omega^2 s)] = \frac{1}{2\pi}e^{i\pi/2}.
\endeq
Using $\det\Psi(s) = W[\Ai(s),\Ai(\omega^2 s)] = \frac{1}{2\pi}e^{i\pi/6}$,
we have
\eq
\Psi^{-1}(s) = 2\pi\bpm \omega^2\Ai'(\omega^2 s) & -\Ai(\omega^2 s) \\ -Ai'(s)e^
{-i\pi/3} & \Ai(s)e^{-i\pi/3} \epm.\label{Psi-inverse}
\endeq
Using $\frac{2}{3}\lambda(z)^{3/2} = \alpha(z)$, equations
(\ref{Airy-asymptotics}) and (\ref{Airy-prime-asymptotics}) yield
\eq
\Ai(q^{2/3}\lambda(z)) = \frac{e^{-q\alpha}}{2\sqrt{\pi}(q^{2/3}\lambda)^{1/4}}\left(1-\frac{c_1}{q\alpha}+O\left(\frac{1}{|q\alpha|^2}\right)\right)\label{Airy-asymptotics-q}
\endeq

\eq
\Ai'(q^{2/3}\lambda(z)) = -\frac{(q^{2/3}\lambda)^{1/4}e^{-q\alpha}}{2\sqrt{\pi}}\left(1-\frac{d_1}{q\alpha}+O\left(\frac{1}{|q\alpha|^2}\right)\right)\label{Airy-prime-asymptotics-q}
\endeq

\eq
\Ai(\omega^2 q^{2/3}\lambda(z)) = \frac{e^{i\pi/6}e^{q\alpha}}{2\sqrt{\pi}(q^{2/3}\lambda)^{1/4}}\left(1+\frac{c_1}{q\alpha}+O\left(\frac{1}{|q\alpha|^2}\right)\right)\label{Airy-asymptotics-w2q}
\endeq

\eq
\omega^2\Ai'(\omega^2 q^{2/3}\lambda(z)) = -\frac{\omega^2 (q^{2/3}\lambda)^{1/4}e^{q\alpha}}{2\sqrt{\pi}e^{i\pi/6}}\left(1+\frac{d_1}{q\alpha}+O\left(\frac{1}{|q\alpha|^2}\right)\right)\label{Airy-prime-asymptotics-w2q}.
\endeq
We insert the asymptotics  \eqref{Airy-asymptotics-q}-\eqref{Airy-prime-asymptotics-w2q} into \eqref{Psi-inverse} resulting in the asymptotic formulas for  $0<\mathrm{arg}(q^{2/3}\lambda(z))<\pi$ and $q\alpha$ large,
\eq
\label{Psi-inverse-asymptotic}
\Psi^{-1}(q^{2/3}\lambda(z))=\sqrt{\pi}e^{i\pi/6}e^{q\alpha\sigma_3}\left\{\bpm 1 & -1\\-i& -i\epm +\frac{1}{q\alpha}\bpm d_1 & -c_1\\ id_1 & ic_1\epm+O\left(\frac{1}{|q\alpha|^2}\right)\right\} \left(\frac{3}{2}q\alpha\right)^{\sigma_3/6}.
\endeq
It is straightforward to compute an analogous expansion for $\Psi^{-1}(s)$ for the other values of $\arg(q^{2/3}\lambda(z))$ in equation \eqref{Psi-definition}.  To do this one must use the asymptotic formulas \eqref{Airy-asymptotics} and \eqref{Airy-prime-asymptotics} as well as the additional expansions (10.4.60) and (10.4.62) from Abramowitz and Stegun \cite{Abramowitz:1965-book}.  Namely, for $|\arg(s)|<2\pi/3$,
\eq \mathrm{Ai}(-s)=\frac{s^{-1/4}}{\sqrt{\pi}}\left[ \sin\left(\frac{2}{3}s^{3/2}+\frac{\pi}{4}\right)-\frac{3c_1}{2s^{3/2}}\cos\left(\frac{2}{3}s^{3/2}+\frac{\pi}{4}\right)+O\left(\frac{1}{|s|^3}\right)\right]\, ,\label{Airy-asymptotics2}\endeq
\eq \mathrm{Ai}(-s)=-\frac{s^{1/4}}{\sqrt{\pi}}\left[ \cos\left(\frac{2}{3}s^{3/2}+\frac{\pi}{4}\right)-\frac{3d_1}{2s^{3/2}}\sin\left(\frac{2}{3}s^{3/2}+\frac{\pi}{4}\right)+O\left(\frac{1}{|s|^3}\right)\right]\, .
\label{Airy-prime-asymptotics2}\endeq

\noindent  After carrying out this computation, the first two terms
in the expansion are the same in all four regions.  In other words,
\eqref{Psi-inverse-asymptotic} is valid not only for
$0<\mathrm{arg}(q^{2/3}\lambda(z))<2\pi/3$ but for all regions in
the definition of $\Psi$ in \eqref{Psi-definition}. Inserting the
expansion in \eqref{Psi-inverse-asymptotic}, equation
\eqref{vR-boundaryO} reduces to

\eq
v_{R} = I + \frac{1}{2q\alpha}\bpm d_1\beta^2+c_1\beta^{-2} & i(d_1\beta^2-c_1\beta^{-2}) \\ i(d_1\beta^2-c_1\beta^{-2}) & -(d_1\beta^2+c_1\beta^{-2}) \epm + O\left(\frac{1}{|q\alpha|^2}\right),
\endeq
for all $z\in\partial\mathcal{O}_{\xi}$.
\end{proof}

Now we explicitly evaluate $R^{(1)}$.

\begin{lemma}\label{lem:R1explicit}
We have
\begin{equation}
R^{(1)} = \bpm \frac{1}{8q(\gamma-1)} - \frac{1}{24q\gamma} &
\frac{1}{8q(\gamma-1)^{1/2}} - \frac{(\gamma-1)^{1/2}}{24q\gamma} \\
\frac{1}{8q(\gamma-1)^{1/2}} - \frac{(\gamma-1)^{1/2}}{24q\gamma} &
-\frac{1}{8q(\gamma-1)} + \frac{1}{24q\gamma} \epm.
\end{equation}
\end{lemma}

\begin{proof}
From Lemma \ref{expansion-of-vR}, it is sufficient to compute the
integrals \eq I_1=\frac{1}{2\pi
i}\int_{\partial\mathcal{O}_\xi}\frac{\beta(s)^2}{\alpha(s)s}\d s
\quad \text{and} \quad I_2=\frac{1}{2\pi
i}\int_{\partial\mathcal{O}_\xi}\frac{1}{\beta(s)^2\alpha(s)s}\d s.
\endeq
We will use the relations $\xi=e^{i\theta_c}$ and
\eq
\label{sin-of-theta-ito-gamma} \sin(\theta_c) =
\frac{2(\gamma-1)^{1/2}}{\gamma}, \quad \cos(\theta_c) =
\frac{\gamma-2}{\gamma}, \quad \sin\left(\frac{\theta_c}{2}\right) =
\frac{1}{\gamma^{1/2}}, \quad \cos\left(\frac{\theta_c}{2}\right) =
\left(\frac{\gamma-1}{\gamma}\right)^{1/2}.
\endeq
Note that $\alpha(z) = \frac{2}{3}(z-\xi)^{3/2}G(z)$ for an analytic
function $G(z)$ in $\mathcal{O}_\xi$ (see the bottom line at p.1157
of \cite{Baik:1999}). Hence by residue calculations, \eq \label{I1}
I_1 = \frac{1}{2\pi
i}\int_{\partial\mathcal{O}_\xi}\frac{3}{2(z-\xi)(z-\overline{\xi})^{1/2}G(z)z}\d
z = \frac{3}{2}\frac{1}{(\xi-\overline{\xi})^{1/2}G(\xi)\xi}
\endeq
and
\eq\begin{split}
I_2 & = \frac{1}{2\pi i}\int_{\partial\mathcal{O}_\xi}\frac{3(z-\overline{\xi})^{1/2}}{2(z-\xi)^2G(z)z}\d z \quad\! = \quad\!  \frac{3}{2}\left(\frac{\d }{\d z}\left.\left[\frac{(z-\overline{\xi})^{1/2}}{G(z)z}\right]\right)\right|_{z=\xi} \\
\label{I2}
    & =  \frac{3}{2}\left[\frac{1}{2(\xi-\overline{\xi})^{1/2}G(\xi)\xi}-\frac{(\xi-\overline{\xi})^{1/2}G'(\xi)}{G(\xi)^2\xi} - \frac{(\xi-\overline{\xi})^{1/2}}{G(\xi)\xi^2}\right].
\end{split}\endeq
But since $\alpha(z) =  \frac{2}{3}(z-\xi)^{3/2}G(z)$ and
$\alpha'(z)=-\frac{\gamma}4\frac{z+1}{z^2}\sqrt{(z-\xi)(z-\overline{\xi})}$,
a straightforward computation yields that \eq \label{Hofxi} G(\xi) =
\lim_{z\rightarrow\xi}\frac{\alpha'(z)}{(z-\xi)^{1/2}} =
-\frac{\gamma}{4}\frac{\xi+1}{\xi^2}(\xi-\overline{\xi})^{1/2}.
\endeq
and \eq\begin{split} G'(\xi) \label{Hprimeofxi}
        & =  \frac{3\gamma}{20}\left(\frac{(\xi+2)(\xi-\overline{\xi})^{1/2}}{\xi^3}-\frac{\xi+1}{2\xi^2(\xi-\overline{\xi})^{1/2}}\right).
\end{split}\endeq
Using~\eqref{sin-of-theta-ito-gamma}, we obtain
\eq
\begin{split}
I_1 &  = \frac{-3}{4(\gamma-1)}+\frac{3}{4(\gamma-1)^{1/2}}i,
\end{split}\endeq
and \eq\begin{split} I_2 & =
\frac{3}{4(\gamma-1)}-\frac{3}{5\gamma}+\left(\frac{3}{4(\gamma-1)^{1/2}}-\frac{3(\gamma-1)^{1/2}}{5\gamma}\right)i.
\end{split}\endeq
Therefore, \eq R^{(1)}_{11} = -R^{(1)}_{22} =
\frac{1}{q}\Re\left[-\frac{7}{72}I_1+\frac{5}{72}I_2\right] =
\frac{1}{8q(\gamma-1)}-\frac{1}{24q\gamma}
\endeq
and \eq R^{(1)}_{12} = R^{(1)}_{21} =
\frac{1}{q}\Im\left[\frac{7}{72}I_1 + \frac{5}{72}I_2\right] =
\frac{1}{8q(\gamma-1)^{1/2}} - \frac{(\gamma-1)^{1/2}}{24q\gamma}.
\endeq
\end{proof}

\subsection{Bound on $\bs{R^{(2)}= \Re\left[\frac{1}{\pi i}\int_{\partial\mathcal{O}_\xi}\mu(s)\cdot O\left(\frac{1}{|q\alpha|^2}\right)\frac{\d s}{s}\right]}$}

We begin by establishing the leading term of $\alpha(z)$ for $z$
near $\xi$.

\begin{lemma} \label{R2-lemma}
For $1\le q<2t$ and for $z$ such that $|z-\xi|\le \min\{\frac12,
|\xi-\overline{\xi}|\}$, \label{alpha-expansion} \eq
\label{alpha-expansion-result}
   \bigg| \alpha(z) -
   \frac{2}{3}(\gamma-1)^{3/4}(z-\xi)^{3/2}e^{-3i\pi/4}e^{-3i\theta_c/2}\bigg|
    \le  \frac{50(\gamma-1)^{3/4}|z-\xi|^{5/2}}{|\xi-\overline{\xi}|}.
\endeq
\end{lemma}

\begin{proof}
Write $z=\xi(1+\e)$. Then $|\epsilon|=|z-\xi|\le  \min\{\frac12,
|\xi-\overline{\xi}|\}$. Under the change of variables $s=\xi(1+\e
u)$, equation (\ref{alpha}) for $\alpha(z)$ becomes \eq
\label{alpha-new-variable} \alpha(z(\e)) =
   -\frac{\gamma\e^{3/2}(1+\xi)\sqrt{\xi-\overline{\xi}}}{4\sqrt{\xi}}
   \int_0^1 \sqrt{u}\left(1+\frac{\xi\e}{\xi-\overline{\xi}}u\right)^{1/2}
   \frac{1+\frac{\xi\e}{1+\xi}u}{(1+\e u)^2}\d u.
\endeq
Using $\xi=e^{i\theta_c}$ and~\eqref{sin-of-theta-ito-gamma}, we
have
\begin{equation}\label{alpha-expansion-eq1}
\begin{split}
    -\frac{\gamma\e^{3/2}(1+\xi)\sqrt{\xi-\overline{\xi}}}{4\sqrt{\xi}}
    &= (\gamma-1)^{3/4}e^{-3i\pi/4}
    \epsilon^{3/2} =
    (\gamma-1)^{3/4}(z-\xi)^{3/2}e^{-3i\pi/4}e^{-3i\theta_c/2}.
\end{split}
\end{equation}
For the integrand in~\eqref{alpha-new-variable}, using the
inequalities $|(1+w)^{1/2}-1|\le |w|$ for $|w|\le 1$ and
$|(1+w)^{-2}-1|\le 10|w|$ for $|w|\le \frac1{2}$, and using the fact
that $\frac1{|1+\xi|}\le \frac2{|\xi-\overline{\xi}|}$ and $1\le
\frac2{|\xi-\overline{\xi}|}$, we obtain \eq
\label{alpha-expansion-eq2}
 \bigg| \left(1+\frac{\xi\e}{\xi-\overline{\xi}}u\right)^{1/2}
   \frac{1+\frac{\xi\e}{1+\xi}u}{(1+\e u)^2}
   -1\bigg|
    \le \frac{50|\e|}{|\xi-\overline{\xi}|}
   = \frac{50|z-\xi|}{|\xi-\overline{\xi}|}.
\endeq
Therefore, we obtain~\eqref{alpha-expansion-result}.
\end{proof}

\begin{lemma}\label{lem:R2}
For $L+1\le q\le 2t-Mt^{1/3}-1$, there is a constant $c>0$ such
that
\begin{equation}
    |R^{(2)}|\le \frac{c(2t)^2}{q^{3/2}(2t-q)^{5/2}}.
\end{equation}
\end{lemma}

\begin{proof}
On $\partial \mathcal{O}_\xi$,
$|z-\xi|=\delta|\xi-\overline{\xi}|\le
\frac1{40}|\xi-\overline{\xi}|$. Hence from Lemma~\ref{R2-lemma}, we
have
\begin{equation}\label{eq:alphaonb}
|\alpha(z)|\geq
\frac{1}{6}(\gamma-1)^{3/4}|z-\xi|^{3/2}= \frac{\delta^{3/2}}6
(\gamma-1)^{3/4}|\xi-\overline{\xi}|^{3/2} =
\frac{4}3\delta^{3/2}\bigg(\frac{\gamma-1}{\gamma}\bigg)^{3/2}
\end{equation}
for $z\in\partial\mathcal{O}_\xi$. Therefore, as $\mu$ and
$\frac{1}{s}$ are bounded on $\partial\mathcal{O}_\xi$, \eq
\label{bound-on-R2} |R^{(2)}| \leq c'\int_{\partial\mathcal{O}_\xi}
\frac{\gamma^3}{q^2(\gamma-1)^3}|\d s| = \frac{2\pi
c'\gamma^3\delta|\xi-\overline{\xi}|}{q^2(\gamma-1)^3} =
\frac{c\gamma^2}{q^2(\gamma-1)^{5/2}}
\endeq
for some constants $c', c>0$, as $|\xi-\overline{\xi}| =
\frac{4(\gamma-1)^{1/2}}{\gamma}$.

\end{proof}

\subsection{Bound on
$\bs{R^{(3)} = \Re\left[\frac{1}{\pi
i}\int_{C_1^+}\mu(s)(v_{R}(s)-I)\frac{\d s}{s}\right]}$}
\label{evaluating-R3}

Since $\mu(s)$ and $1/s$ are bounded on $C_1^+$, $|R^{(3)}|\leq
c'||v_R-I||_{L^1(C_1^+)}$ for some constant $c'>0$. But on $C_1^+$,
$v_R(z)-I = O(e^{-2q\alpha(z)})$. Hence \eq \label{L1-bound-on-R3}
|R^{(3)}|\leq c||e^{-2q\alpha(z)}||_{L^1(C_1^+)},
\endeq for some constant $c>0$. For $z=e^{i\theta}\in C_+$ (hence
$\theta_c< \theta\le \pi$), using \eq
\sqrt{(e^{i\phi}-\xi)(e^{i\phi}-\overline{\xi})} =
|(e^{i\phi}-e^{i\phi_c})(e^{i\phi}-e^{-i\phi_c})|^{1/2}e^{i(\pi+\phi)/2},
\endeq
we have \begin{equation} \begin{split} \alpha(e^{i\theta}) &=
-\frac{\gamma}{4}\int_{\theta_c}^{\theta}
\frac{1+e^{i\phi}}{e^{2i\phi}}\sqrt{(e^{i\phi}-e^{i\theta_c})(e^{i\phi}-e^{-i\theta_c})}
\cdot ie^{i\phi}\d \phi \\
&= \gamma\int_{\theta_c}^\theta\cos\left(\frac{\phi}{2}\right)
\sin^{1/2}\left(\frac{\phi+\theta_c}{2}\right)\sin^{1/2}\left(\frac{\phi-\theta_c}{2}\right)\d\phi.
\end{split}
\end{equation}
(Recall that $\gamma=\frac{2t}{q}$.) Note that $\alpha(\xi)=0$,
$\alpha(s)$ is real and positive on $C_1^+$, and
$\alpha(e^{i\theta})$ increases as $\theta$ increases.

\begin{lemma}\label{lem:alphaR3}
For $1\le q\le 2t$,
\begin{equation}
    \alpha(e^{i\theta}) \ge \frac1{12\pi}\sqrt{\gamma(\gamma-1)}
    (\theta-\theta_c)^2
\end{equation}
for $\theta_c\le \theta\le\pi$.
\end{lemma}

\begin{proof}
We consider two cases separately: $\theta_c\le \frac{\pi}3$ and
$\theta_c\ge \frac{\pi}3$.

Start with the case when $\theta_c\le
\frac{\pi}3$. We consider two sub-cases: $\theta\le \frac{2\pi}3$
and $\theta\ge\frac{2\pi}3$. When $\theta\le \frac{2\pi}3$, $0\le
\frac{\phi}2\le \frac{\pi}3$, $0\le \frac{\phi+\theta_c}2\le
\frac{\pi}2$ and $0\le\frac{\phi-\theta_c}2\le \frac{\pi}3$. Hence
using the basic inequalities $\cos(x)\ge \frac12$ for $0\le x\le
\frac{\pi}3$ and $\sin(x)\ge \frac{2}{\pi}x$ for $0\le x\le
\frac{\pi}2$, we find that
\begin{equation}\label{eq:alpha3mid1}
\begin{split}
    \alpha(e^{i\theta})
    &\ge \frac{\gamma}{2\pi}
    \int_{\theta_c}^\theta
    (\phi+\theta_c)^{1/2}(\phi-\theta_c)^{1/2} d\phi
    \ge \frac{\gamma}{2\pi}
    \int_{\theta_c}^\theta
    (\phi-\theta_c) d\phi = \frac{\gamma}{4\pi}
    (\theta-\theta_c)^2.
\end{split}
\end{equation}
When $\theta\ge\frac{2\pi}3$, from the monotonicity of
$\alpha(e^{i\theta})$ and using~\eqref{eq:alpha3mid1},
\begin{equation}
    \alpha(e^{i\theta})\ge \alpha(e^{\frac{2\pi}3 i})
    \ge \frac{\gamma}{4\pi} \left(\frac{2\pi}3-\theta_c\right)^2.
\end{equation}
For $0\le\theta_c\le \frac{\pi}3$ and $\frac{2\pi}3\le\theta\le
\pi$, we have $\frac{2\pi}3-\theta_c\ge \frac{\pi}3\ge
\frac13(\theta-\theta_c)$. Therefore,
\begin{equation}\label{eq:alpha3mid2}
    \alpha(e^{i\theta})\ge \frac{\gamma}{12\pi} \big(\theta-\theta_c\big)^2.
\end{equation}

For the second case, when $\theta_c\ge \frac{\pi}3$, using the change of
variables $\phi\mapsto \pi-\phi$,
\begin{equation}
    \alpha(e^{i\theta})
    = \gamma\int_{\pi-\theta}^{\pi-\theta_c}
    \sin\left(\frac{\phi}{2}\right)
    \sin^{1/2}\left(\frac{\phi+(\pi-\theta_c)}{2}\right)
    \sin^{1/2}\left(\frac{(\pi-\theta_c)-\phi}{2}\right)\d\phi.
\end{equation}
Note that $0\le \frac{\phi}2\le \frac{\pi}3$, $0\le
\frac{\phi+(\pi-\theta_c)}{2} \le \pi-\theta_c\le \frac{2\pi}3$, and
$0\le \frac{(\pi-\theta_c)-\phi}{2}\le \frac{\pi}3$. Using the basic
inequalities $\sin(x)\ge \frac{3\sqrt{3}}{2\pi}x$ for $0\le x\le
\frac{\pi}3$ and $\sin(x)\ge \frac{3\sqrt{3}}{4\pi}x$ for $0\le x\le
\frac{2\pi}3$, we find that
\begin{equation}
\begin{split}
    \alpha(e^{i\theta})
    &\ge \frac{27 \gamma}{8\sqrt{2}\pi^2} \int_{\pi-\theta}^{\pi-\theta_c}
    \phi (\phi+(\pi-\theta_c))^{1/2}
    ((\pi-\theta_c)-\phi)^{1/2}\d\phi \\
    &\ge \frac{27 \gamma}{8\sqrt{2}\pi^2} \int_{\pi-\theta}^{\pi-\theta_c}
    \phi ((\pi-\theta_c)-\phi)\d\phi \\
    &= \frac{9\gamma}{16\sqrt{2}}
    \big(\pi-\theta_c+2(\pi-\theta)\big) (\theta-\theta_c)^2
    \ge \frac{9\gamma}{16\sqrt{2}\pi^2}
    (\pi-\theta_c) (\theta-\theta_c)^2.
\end{split}
\end{equation}
Since $\sqrt{\frac{1-\gamma}{\gamma}} =
\cos\left(\frac{\theta_c}{2}\right) =
\sin\left(\frac{\pi-\theta_c}{2}\right) \le \frac{\pi-\theta_c}{2}$,
we have
\begin{equation}\label{eq:alpha3mid3}
    \alpha(e^{i\theta}) \ge \frac{9}{8\sqrt{2}\pi^2}
    \sqrt{\gamma(\gamma-1)} (\theta-\theta_c)^2.
\end{equation}
Combining~\eqref{eq:alpha3mid1}, ~\eqref{eq:alpha3mid2},
and~\eqref{eq:alpha3mid3} completes the proof.
\end{proof}

\begin{lemma}\label{R3-does-not-contribute}
For $L+1\le q\le 2t-Mt^{1/3}-1$, there is a constant $c>0$ such
that
\begin{equation}\label{eq:R3est}
    |R^{(3)}|\le \frac{c(2t)^2}{q^{3/2}(2t-q)^{5/2}}.
\end{equation}
\end{lemma}

\begin{proof}
Let $e^{i\theta_*}$ be the endpoint of $C_1^+$ on
$\partial{\mathcal{O}_\xi}$. Note that since radius of
$\partial{\mathcal{O}_\xi}$ is $\delta|\xi-\overline{\xi}|=4\delta
\frac{\sqrt{\gamma-1}}{\gamma}$,
\begin{equation}
    \theta_*-\theta_c\ge 4\delta\frac{\sqrt{\gamma-1}}{\gamma}.
\end{equation}
Using Lemma~\ref{lem:alphaR3} and changing variables,
\begin{equation}
\begin{split}
    \|e^{-2q\alpha}\|_{L^1(C_1^+)} \le \int_{\theta_*}^\pi
    e^{-\frac{q\sqrt{\gamma(\gamma-1)}}{6\pi} (\theta-\theta_c)^2} d\theta
    \le \bigg(\frac{3\pi}{q\sqrt{\gamma(\gamma-1)}}\bigg)^{1/2}
    \int_{x_*}^\infty e^{-\frac12 x^2} dx
\end{split}
\end{equation}
where
\begin{equation}
    x_*= \bigg(\frac{q\sqrt{\gamma(\gamma-1)}}{3\pi}\bigg)^{1/2}
    (\theta_*-\theta_c) \ge \frac{4\delta}{\sqrt{3\pi}}
    \sqrt{q}\bigg(\frac{\gamma-1}{\gamma}\bigg)^{3/4}.
\end{equation}
Using the inequality $\int_a^\infty e^{-\frac12 x^2}dx\le
\frac1{a^3}$ for $a>0$,
\begin{equation}\label{eq:131}
\begin{split}
    \|e^{-2q\alpha}\|_{L^1(C_1^+)} \le
    \bigg(\frac{3\pi}{q\sqrt{\gamma(\gamma-1)}}\bigg)^{1/2}\frac1{x_*^3}
    \le
    \frac{9\pi^2}{128\delta^3}\frac{(2t)^2}{q^{3/2}(2t-q)^{5/2}}.
\end{split}
\end{equation}
Hence from~\eqref{L1-bound-on-R3} we obtain~\eqref{eq:R3est}.
\end{proof}

\subsection{Bound on $\bs{R^{(4)}
 = \Re\left[\frac{1}{\pi i}\int_{C_\text{in}^+\cup C_\text{out}^+}\mu(s)(v_{R}(s)-I)\frac{\d s}{s}\right]}$}
\label{evaluating-R4}

As before, since on $C_\text{in}^+\cup C_\text{out}^+$ the functions $\mu(s)$
and $\frac{1}{s}$ are uniformly bounded and $v_R(z)-I =
O(e^{-2q\alpha(z)})$, \eq \label{R4-bound-by-L1} |R^{(4)}|\leq
c'||v_R-I||_{L^1(C_\text{in}^+\cup C_\text{out}^+)} \le c\|
e^{2q\alpha}\|_{L^1(C_\text{in}^+\cup C_\text{out}^+)}
\endeq
for some constants $c', c>0$.

\begin{lemma}\label{R4-does-not-contribute}
For $L+1\le q\le 2t-Mt^{1/3}-1$, there is a constant $c>0$ such
that
\begin{equation}\label{eq:R4est}
    |R^{(4)}|\le \frac{c(2t)^2}{q^{3/2}(2t-q)^{5/2}}.
\end{equation}
\end{lemma}

\begin{proof}
Let $\gamma_0=\text{csc}^2 (\frac{\pi}{24})$. Let $\delta_4>0$ be a
small positive number defined on p. 1152 of \cite{Baik:1999}. We
estimate $e^{-2q\alpha}$ in the following three cases separately:
(i) $2t(1+\delta_4)^{-1}\le q\le 2t-Mt^{1/3}+1$, (ii)
$\frac{2t}{\gamma_0} \le q\le 2t(1+\delta_4)^{-1}$ and (iii) $L+1\le
q\le \frac{2t}{\gamma_0}$.

(i) For $2t(1+\delta_4)^{-1} \le q\le 2t-Mt^{1/3}-1$, from (6.37) of
\cite{Baik:1999}, there are constants $c_1, c_2, c_3, c_4, c_5>0$
such that \eq || e^{2q\alpha}||_{L^1(C_\text{in}^+\cup
C_\text{out}^+)} \leq c_1\int_{c_2\sqrt{\gamma-1}}^\infty
e^{-c_3qx^3}\d x + c_4e^{-c_5 q}.
\endeq
Using the change of variables $y=c_3qx^3$,
\begin{equation}
\begin{split}
|| e^{2q\alpha}||_{L^1(C_\text{in}^+\cup C_\text{out}^+)} &\leq
\frac{c_1}{3c_3^{1/3}q^{1/3}}\int_{c_2q(\gamma-1)^{3/2}}^\infty
\frac{e^{-y}}{y^{2/3}}\d y  + c_4e^{-c_5 q} \\
&\le \frac{c_1}{3c_2^2c_3q(\gamma-1)}e^{-c_2^3c_3q(\gamma-1)^{3/2}}
+ c_4e^{-c_5 q}\\
&\le \frac{c_1}{3c_2^5c_3^2q^2(\gamma-1)^{5/2}} +
\frac{c_4}{c_5^2q^2}.
\end{split}
\end{equation}
Since $\gamma\ge 1$ and $\gamma-1\le 1+\delta_4$, we find that
\begin{equation}
\begin{split}
    || e^{2q\alpha}||_{L^1(C_\text{in}^+\cup C_\text{out}^+)}
    \le \frac{c\gamma^2}{q^2(\gamma-1)^{5/2}} = \frac{c(2t)^2}{q^{3/2}(2t-q)^{5/2}}
\end{split}
\end{equation}
for a constant $c>0$.

(ii) For $\frac{2t}{\gamma_0}\le q\le 2t(1+\delta_4)^{-1}$, note
that the radius of $\mathcal{O}_\xi$ is of $O(1)$. Hence a standard
calculation in Riemann-Hilbert steepest-descent analysis shows that
for $z\in C^+_{\text{in}}\cap C^+_{\text{out}}$, $\Re (\alpha(z))\le
-c'$ for some constant $c'>0$. Since the length of
$L^1(C_\text{in}^+\cup C_\text{out}^+)$ is bounded,
\begin{equation}\label{eq:R3middle}
    || e^{2q\alpha}||_{L^1(C_\text{in}^+\cup C_\text{out}^+)}
    \le c''e^{-2c'q}
    \le \frac{c''}{(2c')^2q^2}
    \le \frac{c'''\gamma^2}{q^2(\gamma-1)^{5/2}} = \frac{c'''(2t)^2}{q^{3/2}(2t-q)^{5/2}}
\end{equation}
for some constants $c'', c'''>0$.

(iii) Consider the case when $L+1\le q\le \frac{2t}{\gamma_0}$. Then
$0\le \theta_c\le \frac{\pi}{12}$. In this case, we make a specific
choice of $C^+_{\text{in}}$ and $C^+_{\text{out}}$:
\begin{equation}
\begin{split}
    C^+_{\text{in}} &= \bigg\{ \xi+\rho \sin\theta_c e^{i(\theta_c+\frac{7}6\pi)} :
    2\delta \le \rho \le \frac{1}{-\sin(\theta_c+\frac{7}6\pi)}
    \bigg\} \\
    C^+_{\text{out}} &= \bigg\{ \xi+\rho \sin_c e^{i(\theta_c-\frac{1}6\pi)} :
    2\delta\le \rho \le \frac{1}{-\sin(\theta_c-\frac{1}6\pi)}
    \bigg\}.
\end{split}
\end{equation}
The contours $C^+_{\text{in}}$ are straight line segments from $\xi$
to a point on the positive real axis. (Recall that $\mathcal{O}_\xi$
has the radius $\delta|\xi-\overline{\xi}|=2\delta \sin\theta_c$.)
Now we estimate $\Re(\alpha(z))$ for $z\in C^+_{\text{in}}\cup
C^+_{\text{out}}$. For $z\in C^+_{\text{in}}$, take the contour
in~\eqref{alpha} to be the straight line from $\xi$ to $z$. Then
one can check from the geometry that
\begin{equation}
\begin{split}
    &0\le \arg(1+s)\le \theta_c, \qquad 0\le \arg(s)\le \theta_c,
    \qquad \arg(s-\xi)= \theta_c+\frac{7\pi}6\\
    &\frac{\pi}2\le \arg(s-\overline{\xi}) \le \frac{5\pi}6-\theta_c,
    \qquad \arg(ds)= \theta_c+\frac{2\pi}6.
\end{split}
\end{equation}
Therefore the argument of the integrand in~\eqref{alpha} is in
$[2\pi, 2\pi+\frac{\pi}6+2\theta_c]\subset [2\pi, 2\pi+
\frac{\pi}3]$ since $0\le \theta_c\le \frac{\pi}{12}$. Thus, the
cosine of the argument is greater than or equal to
$\cos(\frac{\pi}3)=\frac12$. Therefore, for $z= \xi+\rho
\sin\theta_c e^{i(\theta_c+\frac76\pi)}\in C^+_{\text{in}}$, by
the change of variables $s=\xi+y \sin\theta_c
e^{i(\theta_c+\frac{7}6\pi)}$,
\begin{equation}
\begin{split}
    \Re (\alpha(z))
    &\le -\frac{\gamma}8 \int_{\xi}^z
    \bigg|\frac{1+s}{s^2} \sqrt{(s-\xi)(s-\overline{\xi})}\bigg|
    |ds|\\
    &\le
    -\frac{\gamma\sin^2\theta_c\cos(\frac{\theta_c}2)}{2\sqrt{2}}\int_0^\rho
    \sqrt{y} \frac{|1+\frac{y\sin\theta_c}{2\cos(\frac{\theta_2}2)}e^{i(\frac{\theta_c}2+\frac{7\pi}6)}|}
    {|1+y\sin\theta_ce^{i\frac{7\pi}6}|^2}
    \big|1-i\frac12 ye^{i(\theta_c+\frac{7\pi}6)}\big|^{1/2}dy.
\end{split}
\end{equation}
Using the inequality $|1+xe^{i\phi}|\ge |\sin\phi|$ for all $x\in
\mathbb{R}$, and using $0\le \theta_c\le \frac{\pi}{12}$ and $|y|\le
\frac1{-\sin(\theta_c+\frac{7\pi}6)}\le 2$, we have
\begin{equation}\label{eq:R3L1}
\begin{split}
    \Re (\alpha(z))
    &\le
    -\frac{\gamma\sin^2\theta_c\cos(\frac{\theta_c}2)}{36\sqrt{2}}\int_0^\rho
    \sqrt{y} dy\\
    &= -\frac{\sqrt{2}}{27}\bigg(\frac{\gamma-1}{\gamma}\bigg)^{3/2}
    \rho^{3/2}\\
    &\le
    -\frac{\sqrt{2}}{27}\bigg(\frac{\gamma-1}{\gamma}\bigg)^{3/2}
    (2\delta)^{3/2}
\end{split}
\end{equation}
for $z\in C^+_{\text{in}}$. For $z\in C^+_{\text{out}}$, taking the
contour in~\eqref{alpha} to be the straight line from $\xi$ to $z$,
we can check that
\begin{equation}
\begin{split}
    &0\le \arg(1+s)\le \theta_c, \qquad 0\le \arg(s)\le \theta_c,
    \qquad \arg(s-\xi)= \theta_c-\frac{\pi}6\\
    &\frac{\pi}6-\theta_c \le \arg(s-\overline{\xi}) \le \frac{\pi}2,
    \qquad \arg(ds)= \theta_c-\frac{\pi}6.
\end{split}
\end{equation}
Hence the argument of the integrand in~\eqref{alpha} is in
$[-\theta_c-\frac{\pi}6, 2\theta_c]\subset [-\frac{\pi}4,
\frac{\pi}6]$. Therefore, for $z= \xi+\rho \sin\theta_c
e^{i(\theta_c-\frac16\pi)}\in C^+_{\text{out}}$,
\begin{equation}\label{eq:R3L2}
\begin{split}
    \Re (\alpha(z))
    &\le -\frac{\gamma}{4\sqrt{2}} \int_{\xi}^z
    \bigg|\frac{1+s}{s^2} \sqrt{(s-\xi)(s-\overline{\xi})}\bigg|
    |ds|\\
    &\le
    -\frac{\gamma\sin^2\theta_c\cos(\frac{\theta_c}2)}{2}\int_0^\rho
    \sqrt{y} \frac{|1+\frac{y\sin\theta_c}{2\cos(\frac{\theta_2}2)}e^{i(\frac{\theta_c}2-\frac{\pi}6)}|}
    {|1+y\sin\theta_ce^{-i\frac{\pi}6}|^2}
    \big|1-i\frac12 ye^{i(\theta_c-\frac{\pi}6)}\big|^{1/2}dy \\
    &\le -\frac1{27\sqrt{2}}
    \bigg(\frac{\gamma-1}{\gamma}\bigg)^{3/2} \rho^{3/2} \\
    &\le -\frac1{27\sqrt{2}}
    \bigg(\frac{\gamma-1}{\gamma}\bigg)^{3/2} (2\delta)^{3/2}.
\end{split}
\end{equation}
From~\eqref{eq:R3L1} and~\eqref{eq:R3L2}, arguing as
in~\eqref{eq:R3middle}, we obtain
\begin{equation}\label{eq:R4middle}
    || e^{2q\alpha}||_{L^1(C_\text{in}^+\cup C_\text{out}^+)}
    \le \frac{c''(2t)^2}{q^{3/2}(2t-q)^{5/2}}
\end{equation}
for a constant $c''>0$. Hence we obtain the estimate for
$|R^{(4)}|$.
\end{proof}

\subsection{Bound on $\bs{R^{(5)} = \Re\left[\frac{1}{\pi i}\int_{\partial\mathcal{O}_\xi}(\mu(s)-I)\frac{1}{q\alpha}v_{R}^{q\alpha}\frac{\d s}{s}\right]}$}

\begin{lemma}\label{lem:R5}
For $L+1\le q\le 2t-Mt^{1/3}-1$, there is a constant $c>0$ such
that
\begin{equation}\label{eq:R5est}
    |R^{(5)}|\le \frac{c(2t)^2}{q^{3/2}(2t-q)^{5/2}}.
\end{equation}
\end{lemma}

\begin{proof}
As $\mu-I = (1-\cauchy_R)^{-1}\cauchy_RI$, and as $(1-\cauchy)^{-1}$
and $\cauchy_-$ are uniformly bounded, \eq ||\mu-I||_{L^2(\Sigma)}\leq
c_0||\cauchy_RI||_{L^2(\Sigma)} =
c_0||\cauchy_-(v_{R}-I)||_{L^2(\Sigma)}\leq
c_1||v_{R}-I||_{L^2(\Sigma)}
\endeq
for some constants $c_0, c_1>0$ when $t$ is large enough.  Below, we
assume that $t$ is large enough so that the above estimate
holds. Now
\begin{equation} \begin{split}
|R^{(5)}|&\leq
\int_{\partial\mathcal{O}_\xi}\left|(\mu(s)-I)\frac{v_{R}^{q\alpha}}{q\alpha}
\right| \frac{|\d s|}{|s|}\\
&\le 2 ||\mu-I||_{L^2(\partial\mathcal{O}_\xi)}
\left|\left|\frac{v_{R}^{q\alpha}}{q\alpha}\right|\right|_{L^2(\partial\mathcal{O}_\xi)}\\
&\leq
2c_1||v_R-I||_{L^2(\Sigma)}\left|\left|\frac{v_{R}^{q\alpha}}{q\alpha}\right|\right|_{L^2(\partial\mathcal{O}_\xi)}.
\end{split}
\end{equation}

Since $\beta(z)=\big(\frac{z-\xi}{z-\overline{\xi}}\big)^{1/4}$ is
bounded above and below for $z\in \mathcal{O}_\xi$,
$v_R^{q\alpha}(z)$ in Lemma~\ref{expansion-of-vR} is bounded.
Using~\eqref{eq:alphaonb} and the fact that the radius of
$\mathcal{O}_\xi$ is $\delta|\xi-\overline{\xi}|=
\frac{4\delta\sqrt{\gamma-1}}{\gamma}$, we have
\begin{equation}\label{eq:148}
    \left|\left|\frac{v_{R}^{q\alpha}}{q\alpha}\right|\right|_{L^2(\partial\mathcal{O}_\xi)}^2
    \le \frac{c_2\gamma^2}{q^2(\gamma-1)^{5/2}}
    = \frac{c_2(2t)^{2}}{q^{3/2}(2t-q)^{5/2}}
\end{equation}
for a constant $c_2>0$.

Write $||v_R-I||^2_{L^2(\Sigma)} =
||v_R-I||^2_{L^2(\partial\mathcal{O})} +
||v_R-I||^2_{L^2(\Sigma\backslash\partial\mathcal{O})}$. As
$v_R-I=\frac{v^{q\alpha}_R}{q\alpha}+v_R^E$ in
Lemma~\ref{expansion-of-vR} is bounded by $\frac{c_3}{|q\alpha|}$
for a constant $c_3>0$, we have, as in~\eqref{eq:148},
\begin{equation}\label{eq:149}
    ||v_R-I||^2_{L^2(\partial\mathcal{O})}
    = 2||v_R-I||^2_{L^2(\partial\mathcal{O})_{\xi}}
    \le \frac{c_4(2t)^{2}}{q^{3/2}(2t-q)^{5/2}}
\end{equation}
for a constant $c_4>0$. On the other hand,
\begin{equation}
\begin{split}
    ||v_R-I||^2_{L^2(\Sigma\setminus\partial\mathcal{O})}
    &= 2 ||v_R-I||^2_{L^2(C^+_1)} + 2 ||v_R-I||^2_{L^2(C^+_\text{in}\cup
    C^+_{\text{out}})}\\
    &\le c_3\big( ||e^{-2q\alpha}||_{L^2(C^+_1)}
    + ||e^{2q\alpha}||_{L^2(C^+_\text{in}\cup
    C^+_{\text{out}})} \big) \\
    & \le c_3\big( ||e^{-2q\alpha}||_{L^1(C^+_1)}
    + ||e^{2q\alpha}||_{L^1(C^+_\text{in}\cup
    C^+_{\text{out}})} \big)
\end{split}
\end{equation}
for a constant $c_3>0$, since $e^{-2q\alpha}\le 1$ for $z\in
C^+_1$ and $e^{2q\alpha}\le 1$ for $z\in C^+_\text{in}\cup
C^+_{\text{out}}$. Hence from~\eqref{eq:131}
and~\eqref{eq:R3middle}, we have
\begin{equation}\label{eq:151}
\begin{split}
    ||v_R-I||^2_{L^2(\Sigma\setminus\partial\mathcal{O})}
    \le \frac{c_4(2t)^{2}}{q^{3/2}(2t-q)^{5/2}}
\end{split}
\end{equation}
for a constant $c_4>0$. By
combining~\eqref{eq:148},~\eqref{eq:149}, and~\eqref{eq:151}, we
obtain~\eqref{eq:R5est}.
\end{proof}

\subsection{The Airy part}

From
Lemmas~\ref{lem:R2},~\ref{R3-does-not-contribute},~\ref{R4-does-not-contribute},
and~\ref{lem:R5}, we find that, for $L+1\le q\le [2t-Mt^{1/3}-1]$,
there is a constant $c>0$ such that
\begin{equation}\label{eq:Rrestall}
    \bigg|\sum_{i=2}^5 (R^{(i)}_{22}(0)-\sqrt{\gamma-1}R^{(i)}_{21}(0))
    \bigg| \le
    \frac{c(2t)^{2}}{q^{3/2}(2t-q)^{5/2}} +
    \frac{c(2t)^2}{q^2(2t-q)^2}.
\end{equation}

We need the following result.
\begin{lemma}\label{lem:sumqt}
We have
\begin{equation}
    \lim_{M\rightarrow\infty}\limsup_{ t\to\infty}
    \sum_{q=L+1}^{[2t-Mt^{1/3}-1]}
    \frac{(2t)^2}{q^{3/2}(2t-q)^{5/2}}
    =0
\end{equation}
and
\begin{equation}
    \lim_{L\rightarrow\infty}\limsup_{ t\to\infty}
    \sum_{q=L+1}^{[2t-Mt^{1/3}-1]}
    \frac{(2t)^2}{q^2(2t-q)^2}
    =0.
\end{equation}
\end{lemma}

\begin{proof}
We use the following basic inequality. Let $a,b$ be integers. Let
$s(x)$ be a positive differentiable function in an interval $[a-1,
b+1]$ and there is $c\in (a, b)$ such that $s'(x)<0$ for $x\in [a-1,
c)$ and $s'(x)>0$ for $x\in (c,b+1]$. Then
\begin{equation}\label{eq:sumint}
    \sum_{q=a}^b s(q)
    \le \int_{a-1}^{b+1} s(x)dx.
\end{equation}

As a function of $0<q<2t$, $\frac{(2t)^2}{q^{3/2}(2t-q)^{5/2}}$
decreases for $0<q<\frac34t$ and then increases for $\frac34
t<q<2t$. Hence
\begin{equation}
\begin{split}
    &\lim_{M\rightarrow\infty}\limsup_{ t\to\infty}
    \sum_{q=L+1}^{[2t-Mt^{1/3}-1]}
    \frac{(2t)^2}{q^{3/2}(2t-q)^{5/2}} \\
    &\quad \le \lim_{M\rightarrow\infty} \limsup_{ t\to\infty} \int_{L}^{2 t-Mt^{1/3}}
    \frac{(2t)^2}{q^{3/2}(2 t -q)^{5/2}}\d q \\
    &\quad =  \lim_{M\rightarrow\infty}\limsup_{ t\to\infty}\left[
    \frac{4(-3t^2+6t q-2 q^2)}{3t (2 t -q)^{3/2}q^{1/2}}\right]_{L}^{2t-Mt^{1/3}}=0.
\end{split}
\end{equation}
As a function of $0<q<2t$, $\frac{(2t)^2}{q^{2}(2t-q)^{2}}$
decreases for $0<q<t$ and then increases for $t<q<2t$. Hence
\begin{equation}
\begin{split}
    &\lim_{L\rightarrow\infty}\limsup_{ t\to\infty}
    \sum_{q=L+1}^{[2t-Mt^{1/3}-1]}
    \frac{(2t)^2}{q^2(2t-q)^2} \\
&\quad \le \lim_{L\rightarrow\infty}\lim_{ t\to\infty} \int_{L}^{2
t -Mt^{1/3}}
\frac{(2t)^2}{q^2(2 t -q)^2}\d q \\
 &\quad  =  \lim_{L\rightarrow\infty}\lim_{ t\to\infty}
 \left[\frac{-2 t +2q}{q(2 t -q)}-\frac{1}{t }\log\left(\frac{2 t -q}{q}\right)
 \right]_{L}^{2 t -Mt^{1/3}}  =  \quad 0.
\end{split}
\end{equation}
\end{proof}

Now we prove the main result of Section~\ref{airy-part}.
\begin{lemma}[Airy part]\label{lem:Airy}
We have
\begin{equation}\label{eq:Airytermall}
\begin{split}
    &\lim_{L,M\rightarrow\infty} \lim_{t\to\infty}
    \Bigg|\sum_{q=L+1}^{[2 t -Mt^{1/3}-1]} \log(\kappa^{-2}_{q-1})
    -\bigg( t^2-2tL+\frac12 L^2\log(2t)-(\frac12L^2-\frac1{12})\log L
    +\frac34 L^2 \\
    &\qquad\qquad\qquad
    - \frac1{12}M^3-\frac18\log M + \frac1{24}\log 2\bigg)
    \Bigg| =0.
\end{split}
\end{equation}
\end{lemma}

\begin{proof}
We first prove that
\begin{equation} \begin{split} \label{Airy1}
\lim_{L,M\rightarrow\infty}\lim_{ t\to\infty}
&\bigg\{\sum_{q=L+1}^{[2 t -Mt^{1/3}-1]}\hspace{-.15in}
\log(\kappa^{-2}_{q-1}) -\sum_{q=L+1}^{[2 t -Mt^{1/3}-1]}
\bigg(-q\big(-\frac{2t}{q}+\log\big(\frac{2t}{q}\big)+1\big)
+\frac{1}{2}\log\big(\frac{2t}{q}\big) \\
&\quad +\frac{1}{8(2t-q)} +\frac{1}{12q}\bigg)\bigg\}=0.
\end{split}
\end{equation}

Using $1+R^{(1)}_{22}(0)-\sqrt{\gamma-1}R^{(1)}_{21}(0)= 1
-\frac{1}{8(2t-q)} -\frac{1}{12q}\ge \frac12$ and using
Lemmas~\ref{lem:R2},~\ref{R3-does-not-contribute},~\ref{R4-does-not-contribute},
and~\ref{lem:R5}, we find that
\begin{equation}
    \bigg|\frac{\sum_{i=2}^5 (R^{(i)}_{22}(0)-\sqrt{\gamma-1}R^{(i)}_{21}(0))}
    {1+R^{(1)}_{22}(0)-\sqrt{\gamma-1}R^{(1)}_{21}(0)}\bigg| \le
    \frac12
\end{equation}
for $L+1\le q\le [2 t -Mt^{1/3}-1]$, when we take $t$ large
enough. Hence using $|\log (1+x)|\le 2|x|$ for $|x|\ge \frac12$,
\begin{equation}\label{eq:Rrestsum}
\begin{split}
    &\lim_{L,M\rightarrow\infty}\lim_{ t\to\infty}
    \Bigg| \sum_{q=L+1}^{[2 t -Mt^{1/3}-1]}
    \log\bigg(1+\frac{\sum_{i=2}^5 (R^{(i)}_{22}(0)-\sqrt{\gamma-1}R^{(i)}_{21}(0))}
    {1+R^{(1)}_{22}(0)-\sqrt{\gamma-1}R^{(1)}_{21}(0)}\bigg) \Bigg| \\
    &\quad \le \lim_{L,M\rightarrow\infty}\lim_{ t\to\infty}
    4\sum_{q=L+1}^{[2 t -Mt^{1/3}-1]}
    \Bigg|\sum_{i=2}^5
    \big( R^{(i)}_{22}(0)-\sqrt{\gamma-1}R^{(i)}_{21}(0)
    \big)\Bigg| \\
    &\quad \le \lim_{L,M\rightarrow\infty}\lim_{ t\to\infty}
    4c\sum_{q=L+1}^{[2 t -Mt^{1/3}-1]}
    \frac{(2t)^{2}}{q^{3/2}(2t-q)^{5/2}} +
    \frac{(2t)^2}{q^2(2t-q)^2} =0
\end{split}
\end{equation}
using~\eqref{eq:Rrestall} and Lemma~\ref{lem:sumqt}.

On the other hand, from Lemma~\ref{lem:R1explicit},
\begin{equation}
    \sum_{q=L+1}^{[2 t -Mt^{1/3}-1]}
    \log(1+R^{(1)}_{22}(0)-\sqrt{\gamma-1}R^{(1)}_{21}(0))
    = \sum_{q=L+1}^{[2 t -Mt^{1/3}-1]}
    \log \bigg( 1 -\frac{1}{8(2t-q)} -\frac{1}{12q}\bigg).
\end{equation}
Using $-x^2\le \log(1-x)+x\le 0$ for $0\le x\le \frac12$,
\begin{equation}\label{eq:R1firstterm}
\begin{split}
    &\Bigg| \sum_{q=L+1}^{[2 t -Mt^{1/3}-1]}
    \log \bigg( 1 -\frac{1}{8(2t-q)} -\frac{1}{12q}\bigg)
    + \sum_{q=L+1}^{[2 t -Mt^{1/3}-1]}
    \bigg( \frac{1}{8(2t-q)} + \frac{1}{12q} \bigg) \Bigg| \\
    &\le \sum_{q=L+1}^{[2 t -Mt^{1/3}-1]} \bigg(
    \frac{1}{64(2t-q)^2} +\frac1{48(2t-q)q}+ \frac{1}{144q^2}
    \bigg)\\
    &\le \int_{L}^{2t-Mt^{1/3}}
    \bigg(
    \frac{1}{64(2t-q)^2} +\frac{1}{96t}\big(\frac1{2t-q}+ \frac1{q}\big)+ \frac{1}{144q^2}
    \bigg) dq \to 0
\end{split}
\end{equation}
as $t\to\infty$ by evaluating the integral explicitly.

Using~\eqref{eq:Rrestsum} and ~\eqref{eq:R1firstterm}
in~\eqref{eq:73}, we obtain~\eqref{Airy1}.

Now we compute each term of the sum in~\eqref{Airy1}. Note that
$[2t-Mt^{1/3}-1]=2t-Mt^{1/3}-1-\epsilon$ for $0\le \epsilon<1$. We
have
\begin{equation}\label{eq:Airyq1}
\sum_{q=L+1}^{[2 t -Mt^{1/3}-1]} \left(-q\left(-\frac{2t}{q}+1\right)\right) =
\sum_{q=L+1}^{[2 t -Mt^{1/3}-1]} (2t-q) = \frac12 (2t-L-1)(2t-L) -
\frac12 (Mt^{1/3}+\epsilon) (Mt^{1/3}+1+\epsilon)
\end{equation}
and
\begin{equation}\label{eq:Airyq2}
    \sum_{q=L+1}^{[2 t -Mt^{1/3}-1]} (-q+\frac12)\log (2t)
    = -\frac12 ((2t-Mt^{1/3}-1-\epsilon)^2-L^2) \log(2t)
\end{equation}
Since for positive integer $m$
\begin{equation}
     \sum_{k=1}^{m-1}k\log k =  m\log ((m-1)!) -  \sum_{k=1}^{m-1}\log (k!)
     = m\log (\Gamma(m)) - \log G(m+1),
\end{equation}
we find that
\begin{equation}\label{eq:Airyq3}
\begin{split}
    \sum_{q=L+1}^{[2 t -Mt^{1/3}-1]} \left(q-\frac12\right)\log q
    = & \left(2t-Mt^{1/3}-\epsilon-\frac12\right) \log
    \Gamma(2t-Mt^{1/3}-\epsilon) -\log G(2t-Mt^{1/3}-\epsilon+1)\\
    & -\left(L+\frac12\right)\log \Gamma(L+1) +\log G(L+2)
\end{split}
\end{equation}
Note that from Stirling's formula for the Gamma function and the
asymptotcs~\eqref{G-at-large-z} for the Barnes G-function, as
$z\to\infty$,
\begin{equation}\label{eq:GGlarge}
  \left(z+\frac12\right)\log \Gamma(z+1)-\log G(z+2)
  = \left(\frac{z^2}2-\frac16\right)\log z - \frac{z^2}4+\frac{z}2 -
  \frac14\log(2\pi)+\frac1{12} -\zeta'(-1) + o(1).
\end{equation}
Now using the fact that
\begin{equation}\label{eq:harmonic}
  \lim_{K\to\infty} \big|
\sum_{q=1}^{K}\frac{1}{q}- \log{K}-\gamma\big| =0,
\end{equation}
where $\gamma$ is Euler's constant, we obtain
\begin{equation}\label{eq:Airyq4}
\begin{split}
    &\lim_{t\to\infty}
    \Bigg|\sum_{q=L+1}^{[2 t -Mt^{1/3}-1]}\frac1{8(2t-q)}
    - \frac18\log(2t)+\frac18\log(Mt^{1/3}) \Bigg| =0
\end{split}
\end{equation}
and
\begin{equation}\label{eq:Airyq5}
\begin{split}
    &\lim_{t\to\infty}
    \Bigg|\sum_{q=L+1}^{[2 t -Mt^{1/3}-1]}
    \frac1{12q} -
    \frac1{12}\log (2t) -\frac1{12}\gamma +\frac1{12}\sum_{q=1}^L \frac1{q}\Bigg| =0.
\end{split}
\end{equation}
Therefore,
using~\eqref{eq:Airyq1},~\eqref{eq:Airyq2},~\eqref{eq:Airyq3},~\eqref{eq:Airyq4},~\eqref{eq:Airyq5},
and~\eqref{eq:GGlarge}, we obtain
\begin{equation}\label{eq:Aipartqq}
\begin{split}
    &\lim_{ t\to\infty} \Bigg| \sum_{q=L+1}^{[2 t -Mt^{1/3}-1]}
    \left(-q\big(-\frac{2t}{q}+\log\big(\frac{2t}{q}\big)+1\big)
+\frac{1}{2}\log\big(\frac{2t}{q}\big)+\frac{1}{8(2t-q)}
+\frac{1}{12q}\right)\\
&\qquad - \bigg(t^2-2tL+\frac{L^2}{2}\log(2t) \bigg)
    \Bigg| \\
    & = - \frac1{12}M^3-\frac18\log M + \frac1{24}\log
    2  -\frac14\log(2\pi) + \frac1{12}-\zeta'(-1)\\
    &\qquad - \frac12(L+1)
    \left(L+\frac12\right)\log \Gamma(L+1) + \log G(L+2)
    +\frac1{12}\bigg(\gamma-\sum_{q=1}^L \frac1{q}\bigg).
\end{split}
\end{equation}
Hence using~\eqref{eq:GGlarge} and~\eqref{eq:harmonic} again, we
obtain~\eqref{eq:Airytermall}.
\end{proof}

\section{Proof of (\ref{Fdown}) in Theorem~\ref{thm1}: computation of $\boldsymbol{F(x)}$}\label{sec:Fproof}

Recall equation (\ref{sum-parts}) which we rewrite here: \eq
\label{sum-parts-rewrite} \log(e^{-t^2}D_{n}) = -t^2 +
\underbrace{\log(D_L)}_\text{\vspace{.2in}exact part} +
\underbrace{\sum_{q=L+1}^{[2 t -Mt^{1/3}-1]}
\log(\kappa_{q-1}^{-2})}_\text{Airy part} + \underbrace{\sum_{q= [2
t -Mt^{1/3}]}^{[2 t +xt^{1/3}]}
\log(\kappa_{q-1}^{-2})}_\text{Painlev\'e part}.
\endeq
For the Painlev\'e part, from \cite{Baik:1999},
\begin{equation}\label{eq:Ppartqq}
\begin{split}
  \lim_{t\to\infty} \sum_{q= [2t -Mt^{1/3}]}^{[2 t +xt^{1/3}]}
  \log(\kappa_{q-1}^{-2})
  &= -\int_{-M}^x R(y)dy \\
  &= -\int_{-M}^x \bigg( R(y) -\frac14y^2+\frac1{8y}  \bigg) dy
    + \frac1{12}x^3 - \frac18\log|x|\\
    &\qquad  + \frac1{12}M^3+\frac18\log M.
\end{split}
\end{equation}
By combining~\eqref{eq:Ppartqq},~\eqref{eq:Airytermall},
and~\eqref{exact-part-result}, we obtain
\begin{equation}
\begin{split}
    &\lim_{M, L\to\infty} \lim_{t\to\infty} \log(e^{-t^2}D_n(t)) \\
    &= -\int_{-M}^x \bigg( R(y) -\frac14y^2+\frac1{8y}  \bigg) dy
    + \frac1{12}x^3 - \frac18\log|x| + \frac1{24}\log 2 +\zeta'(-1).
\end{split}
\end{equation}

\section{Proof of (\ref{Edown}) in Theorem~\ref{thm1}: computation of
$\boldsymbol{E(x)}$}
\label{E-of-x}

Let \eq I_j(2t) = \frac{1}{2\pi}\int_{-\pi}^\pi
e^{2t\cos\theta}e^{ij\theta}\d\theta.
\endeq
Set (see \cite{Baik:2001a}) \eq D_\ell^{++}(t) =
\det(I_{j-k}(2t)+I_{j+k+2}(2t))_{0\leq j,k \le \ell-1},
\endeq
\eq D_\ell^{-+}(t) = \det(I_{j-k}(2t)+I_{j+k+1}(2t))_{0\leq j,k
\le \ell-1}.
\endeq
It is shown in Corollary 7.2 of \cite{Baik:2001b} that for
\begin{equation}
    \ell=[t+\frac{x}{2}t^{1/3}],
\end{equation}
where $x$ lies in a compact subset of $\mathbb{R}$, we have \eq
\label{D++ell-to-FE}
\lim_{t\to\infty}|e^{-t^2/2}D_{\ell-1}^{++}(t)-F(x)E(x)| = 0,
\endeq
\eq \label{D-+ell-to-FE}
\lim_{t\to\infty}|e^{-t^2/2-t}D_\ell^{-+}(t)-F(x)E(x)| = 0.
\endeq
In \cite{Baik:2001b}, the above results are shown (with
$\widehat{x}$ in place of $x$) for the alternate scaling
$t=\ell-\frac{\widehat{x}}{2}\ell^{1/3}$.  With the above scaling
$\ell=[t+\frac{x}{2}t^{1/3}]$, we find $\widehat{x} =
x(1+\frac{x}{2}t^{-2/3})^{-1/3}$.  Since $x$ is in a compact set,
and $E(x)$ and $F(x)$ are continuous, the above results follow. From
equations~\eqref{D++ell-to-FE} and~\eqref{D-+ell-to-FE}, for a fixed
$x\in\mathbb{R}$, \eq \label{F2E2-ito-Ds} F(x)^2E(x)^2 =
\lim_{t\to\infty}e^{-t^2-t}D_{\ell-1}^{++}D_\ell^{-+}, \qquad
\ell=[t+\frac{x}{2}t^{1/3}].
\endeq

Let $\pi_{j}(z;t)$ be the monic orthogonal polynomial of degree
$k$ with respect to the weight
$\frac{1}{2\pi}e^{2t\cos\theta}\d\theta$ on the unit circle, as
introduced in Section \ref{intro}. It is shown in Corollary 2.7 of
\cite{Baik:2001a} that (cf.~\eqref{eq:produp} above) \eq
\label{D-ell++} e^{-t^2/2}D_\ell^{++}(t)=\prod_{j=\ell}^{\infty}
\kappa^2_{2j+2}(t)(1-\pi_{2j+2}(0;t)) = \prod_{j=\ell}^\infty
\frac{\kappa_{2j+1}^2}{1+\pi_{2j+2}(0;t)},
\endeq
\eq
\label{D-ell-+}
e^{-t^2/2-t}D_\ell^{-+}(t)=\prod_{j=\ell}^{\infty} \kappa_{2j+1}^2(t)(1+\pi_{2j+1}(0;t)) = \prod_{j=\ell}^\infty\frac{\kappa_{2j}^2}{1-\pi_{2j+1}(0;t)},
\endeq
where the last equalities in~\eqref{D-ell++} and~\eqref{D-ell-+}
use the basic identity (see e.g \cite{Szego:1975-book}) \eq
1-\pi_k^2(0;t) = \frac{\kappa_{k-1}^2}{\kappa_k^2}.
\endeq
Using equations (\ref{D-ell++}) and (\ref{D-ell-+}), we can write
\eq
D_{\ell-1}^{++} = D_{L-1}^{++}\prod_{j=L}^{\ell-1}\frac{D_j^{++}}{D_{j-1}^{++}} = D_{L-1}^{++}\prod_{j=L}^{\ell-1}\kappa_{2j-1}^{-2}(t)(1+\pi_{2j}(0;t)),
\endeq
\eq D_\ell^{-+} =
D_L^{-+}\prod_{j=L+1}^\ell\frac{D_j^{-+}}{D_{j-1}^{-+}} =
D_L^{-+}\prod_{j=L+1}^\ell
\kappa_{2j-2}^{-2}(t)(1-\pi_{2j-1}(0;t)),
\endeq
and hence we have \eq D_{\ell-1}^{++}D_\ell^{-+} =
D_{L-1}^{++}D_L^{-+}\prod_{k=2L-1}^{2\ell-2}\kappa_k^{-2}(t)\cdot\prod_{j=L}^{\ell-1}\left[(1+\pi_{2j}(0;t))(1-\pi_{2j+1}(0;t))\right].
\endeq
Using (recall that $D_\ell=\det(I_{j-k}(2t))_{0\le j,k\le \ell-1}$
is the $\ell \times \ell$ Toeplitz determinant given in~\eqref{Dn})
\eq \frac{D_a}{D_b} = \prod_{j=b}^{a-1}\kappa_j^{-2}(t),
\endeq
we find that \eq D_{\ell-1}^{++}D_\ell^{-+} =
D_{L-1}^{++}D_L^{-+}\frac{D_{2\ell-1}}{D_{2L-1}}\prod_{j=L}^{\ell-1}[(1+\pi_{2j}(0;t))(1-\pi_{2j+1}(0;t))].
\endeq
Inserting this equation into~\eqref{F2E2-ito-Ds}, and
using~\eqref{eq:DntoF2}, \eq
\begin{split}
F(x)^2E(x)^2 & = F_2(x)\lim_{t\to\infty}
e^{-t}\frac{D_{L-1}^{++}D_L^{-+}}{D_{2L-1}}
\prod_{j=L}^{\ell-1}[(1+\pi_{2j}(0;t))(1-\pi_{2j+1}(0;t))].
\end{split}
\endeq
Hence we find (cf.~\eqref{eq:proddown2})
\begin{equation}
    E(x)^2= \lim_{t\to\infty}
e^{-t}\frac{D_{L-1}^{++}D_L^{-+}}{D_{2L-1}}
\prod_{j=L}^{\ell-1}[(1+\pi_{2j}(0;t))(1-\pi_{2j+1}(0;t))], \qquad
\ell=\left[t+\frac{x}{2} t^{1/3}\right].
\end{equation}

In analogy to equation~\eqref{sum-parts}, we write
\eq
\begin{split}
\label{parts-of-2logE} 2\log E(x) & = \lim_{L, M\to\infty}
\lim_{t\rightarrow\infty}\Bigg( -t +
\underbrace{\log\frac{D_{L-1}^{++}D_L^{-+}}{D_{2L-1}}}_\text{Exact
part} + \underbrace{\sum_{j=L}^{[t-\frac{M}{2}t^{1/3}]}
\log[(1+\pi_{2j}(0;t))(1-\pi_{2j+1}(0;t))]}_\text{Airy part}\\
   &  \hspace{1.5in}+ \underbrace{\sum_{j=[t-\frac{M}{2}t^{1/3}+1]}^{[t+\frac{x}{2} t^{1/3}-1]}
   \log[(1+\pi_{2j}(0;t))(1-\pi_{2j+1}(0;t))]}_\text{Painlev\'e part} \Bigg).
\end{split}
\endeq
We compute each term as in the case of $\log F(x)$.

\begin{lemma} {\rm (The Painlev\'e part)} \label{painleve-lemma}
We have for $x<0$ and $M>0$, \eq
\begin{split}
&\lim_{t\to\infty}\sum_{j=[t-\frac{M}{2}t^{1/3}+1]}^{[t+\frac{x}{2}
t^{1/3}-1]}\log[(1+\pi_{2j}(0;t))(1-\pi_{2j+1}(0;t))] \\
&\quad = \int_{-M}^x\left(q(y)-\sqrt{\frac{-y}{2}}\right)\d y +
\frac{\sqrt{2}}{3}M^{3/2} - \frac{\sqrt{2}(-x)^{3/2}}{3}.
\end{split}
\endeq
\end{lemma}

\begin{proof}
This follows from the following results in Corollaries 7.1 of
\cite{Baik:2001b}: for $x$ in a compact subset of $\mathbb{R}$,
\eq \label{pi-even-to-E}
\lim_{t\to\infty}\left|\sum_{j=[t+\frac{x}{2}t^{1/3}]}^\infty\log(1+\pi_{2j+2}(0;t))+\log
E(x)\right|=0,
\endeq
\eq \label{pi-odd-to-E}
\lim_{t\to\infty}\left|\sum_{j=[t+\frac{x}{2}t^{1/3}]}^\infty\log(1-\pi_{2j+1}(0;t))+\log
E(x)\right|=0.
\endeq
We remark that in \cite{Baik:2001b}, the notations $v(x)=-R(x)$ and
$u(x)=-q(x)$ are used.
\end{proof}

\begin{lemma} \rm{(The Airy part)}
\label{airy-lemma} \eq
\lim_{L,M\to\infty}\lim_{t\to\infty}\left|\sum_{j=L}^{[t-\frac{M}{2}t^{1/3}]}\log[(1+\pi_{2j}(0;t))(1-\pi_{2j+1}(0;t))]
- \left(t-\frac{\sqrt{2}}{3}M^{3/2}-\left(2L-\frac12\right) \log
2\right)\right| = 0.
\endeq
\end{lemma}

\begin{proof}
Using~\eqref{eq:piR}, \eq
\begin{split}
\label{goe-airy-proof1}
    &\sum_{j=L}^{[t-\frac{M}{2}t^{1/3}]}\log[(1+\pi_{2j}(0))(1-\pi_{2j+1}(0))]
    = \sum_{q=2L}^{2[t-\frac{M}2t^{1/3}]+1}  \log
    \bigg(1+\sqrt{\frac{\gamma-1}{\gamma}}R_{11}(0;q) -
    \frac{1}{\sqrt{\gamma}}R_{12}(0;q) \bigg).
\end{split}
\endeq
Using~\eqref{eq:R1to5} and the fact that
$\sqrt{\gamma-1}R^{(1)}_{11}(0)-R^{(1)}_{12}(0)=0$, which follows from
Lemma~\ref{lem:R1explicit}, this equals \eq\label{eq:EAirytwo}
\begin{split}
    &\sum_{q=2L}^{2[t-\frac{M}2t^{1/3}]+1}
    \log\bigg(1+\sqrt{\frac{\gamma-1}{\gamma}}\bigg) +
    \sum_{q=2L}^{2[t-\frac{M}2t^{1/3}]+1}  \log
    \bigg(1+\frac{\sum_{i=2}^5 (\sqrt{\gamma-1}R^{(i)}_{11}(0;q) -
    R^{(i)}_{12}(0;q))}{\sqrt{\gamma}+\sqrt{\gamma-1}} \bigg).
\end{split}
\endeq
From
Lemmas~\ref{lem:R2},~\ref{R3-does-not-contribute},~\ref{R4-does-not-contribute},
and~\ref{lem:R5}, the same estimate as in~\eqref{eq:Rrestall}
holds for $\sum_{i=2}^5 (\sqrt{\gamma-1}R^{(i)}_{11}(0;q) -
R^{(i)}_{12}(0;q))$ for $L+1\le q\le [2t-Mt^{1/3}-1]$. Therefore
the same argument as in~\eqref{eq:Rrestsum} implies that the
second sum in~\eqref{eq:EAirytwo} vanishes in the limit.
Therefore,
\begin{equation}
    \lim_{L,M\to\infty}\lim_{t\to\infty}
    \left|\sum_{j=L}^{[t-\frac{M}{2}t^{1/3}]}\log[(1+\pi_{2j}(0;t))(1-\pi_{2j+1}(0;t))]
- \sum_{q=2L}^{2[t-\frac{M}2t^{1/3}]+1}
    \log\bigg(1+\sqrt{\frac{2t-q}{2t}}\bigg) \right|=0.
\end{equation}

We use the Euler-Maclaurin summation formula
\begin{equation}
    \sum_{k=a}^b f(k)  =\int_a^b
    f(x)dx + \frac{f(a)+f(b)}{2} + \frac{f'(b)-b'(a)}{6\cdot4!} + \text{Err}, \qquad
    |\text{Err}|\le \frac{2}{(2\pi)^2} \int_a^b |f^{(3)}(x)|dx.
\end{equation}
Set $f(q)=\log\left(1+\sqrt{\frac{2t-q}{2t}}\right)$ and write
$[t-\frac{M}2t^{1/3}]=t-\frac{M}2t^{1/3}-\epsilon$ where $0\le
\epsilon<1$. Then
\begin{equation}\label{eq:198}
    \lim_{t\to\infty} \frac{f(2t-Mt^{1/3}-2\epsilon+1)+f(2L)}{2}
    =\frac12\log 2,
\end{equation}
and
\begin{equation}\label{eq:199}
    \lim_{t\to\infty} \frac{f'(2t-Mt^{1/3}-2\epsilon+1)+f'(2L)}{6\cdot
    4!}=0.
\end{equation}
Also using $f^{(3)}(q)\ge 0$,
\begin{equation}\label{eq:200}
    |\text{Err}|\le \frac{2}{(2\pi)^2} \int_{2L}^{2t-Mt^{1/3}-2\epsilon+1} |f^{(3)}(q)|dq
    = \frac{2}{(2\pi)^2} (f''(2t-Mt^{1/3}-2\epsilon+1)-f''(2L))
    \to 0,
\end{equation}
as $t\to\infty$. Finally, changing variables to
$w=1+\sqrt{\frac{2t-q}{2t}}$ and setting
$\delta_1=1+\sqrt{\frac{t-L}{t}}$ and
$\delta_2=\sqrt{\frac{Mt^{1/3}+2\epsilon -1}{2t}}$,
\begin{equation}\label{goe-airy-proof2}
\begin{split}
    \int_{2L}^{2t-Mt^{1/3}-2\epsilon+1}\log\left(1+\sqrt{\frac{2t-q}{2t}}\right)\d q
    & = 4t\int_{\delta_1}^{1+\delta_2}(1-w)\log(w)\d w \\
    &= 4t \bigg[\left(w-\frac12 w^2\right)\log w +\frac14 w^2-w
    \bigg]_{\delta_1}^{1+\delta_2}\\
    &= t - \frac{\sqrt{2}}{3}M^{3/2} - 2L\log 2 + O\left(\frac{1}{t^{1/3}}\right).
\end{split}
\end{equation}
Hence the result follows.
\end{proof}

\begin{lemma} \rm{(The exact part)}\label{exact-lemma}
We have \eq
\lim_{L\to\infty}\lim_{t\to\infty}\left[\log\frac{D_{L-1}^{++}D_L^{-+}}{D_{2L-1}}
- (2L-1)\log 2 \right] = 0.
\endeq
\end{lemma}

\begin{proof}
Note (see \cite{Baik:2001a}) that \eqarr \notag
D_{L-1}^{++} & = & \det\left[\frac{1}{2\pi}\int_{-\pi}^\pi[e^{i(j-k)\theta}-e^{i(j+k+2)\theta}]e^{2t\cos\theta}\d\theta\right]_{0\leq j,k \leq L-2} \\
   & = & \det\left[\frac{1}{i\pi}\int_{-\pi}^\pi e^{i(j+1)\theta}\sin((k+1)\theta)e^{2t\cos\theta}\d\theta\right]_{0\leq j,k \leq L-2} \\
\notag
   & = & \det\left[\frac{2}{\pi}\int_0^\pi \frac{\sin((j+1)\theta)}{\sin\theta}\cdot\frac{\sin((k+1)\theta)}{\sin\theta}\sin^2\theta e^{2t\cos\theta}\d\theta\right]_{0\leq j,k \leq L-2}.
\endeqarr
Changing variables to $x=\cos\theta$ and noting that
$\frac{\sin((j+1)\theta)}{\sin\theta}$ is a polynomial in $x$ of
the form $2^jx^j+\cdots$ (a constant multiple of the Chebyshev
polynomial of the second kind), we have \eqarr \notag
D_{L-1}^{++} & = & \det\left[\frac{2^{j+k+1}}{\pi}\int_{-1}^1 x^{j+k}\sqrt{1-x^2}e^{2tx}\d x\right]_{0\leq j,k \leq L-2} \\
   & = & \left(\frac{2}{\pi}\right)^{L-1}\frac{2^{(L-1)(L-2)}}{(L-1)!}\int_{[-1,1]^{L-1}}\prod_{1\leq k<\ell\leq L-1}|x_k-x_\ell|^2\prod_{j=1}^{L-1}\sqrt{1-x_j^2}e^{2tx_j}\d x_j.
\endeqarr
A steepest descent analysis yields that \eq
\lim_{t\to\infty}D_{L-1}^{++} \cdot \left(
\frac{e^{2t(L-1)}}{t^{(L-1)^2+(L-1)/2}\pi^{(L-1)}(L-1)!}\int_{[0,\infty]^{L-1}}\prod_{1\leq
k<\ell\leq
L-1}|y_k-y_\ell|^2\cdot\prod_{j=1}^{L-1}\sqrt{y_j}e^{-y_j}\d y_j
\right)^{-1} = 1.
\endeq
The multiple integral is another Selberg integral (corresponding
to a Laguerre ensemble) which is evaluated explicitly as (see
(17.6.5) of  \cite{Mehta:1991-book}) \eq
\int_{[0,\infty]^{L-1}}\prod_{1\leq k<\ell\leq
L-1}|y_k-y_\ell|^2\cdot\prod_{j=1}^{L-1}\sqrt{y_j}e^{-y_j}\d y_j =
(L-1)!\frac{G(L+\frac{1}{2})G(L)}{G(\frac{3}{2})}.
\endeq
Hence \eq \label{DLm1++} \lim_{t\to\infty}D_{L-1}^{++} \cdot
\left( \frac{e^{2t(L-1)}}{t^{(L-1)^2+(L-1)/2}\pi^{(L-1)}} \cdot
\frac{G(L+\frac{1}{2})G(L)}{G(\frac{3}{2})} \right)^{-1} = 1.
\endeq

Similarly,  \eq
\begin{split}
D_L^{-+} & = \; \det\left[\frac{1}{2\pi}\int_{-\pi}^\pi(e^{i(j-k)\theta}+e^{i(j+k+1)\theta})e^{2t\cos\theta}\d\theta\right]_{0 \leq j,k \leq L-1} \\
   & = \; \det\left[\frac{1}{\pi}\int_{-\pi}^\pi e^{i(j+\frac{1}{2})\theta}
   \cos\left(\left(k+\frac{1}{2}\right)\theta\right)e^{2t\cos\theta}\d\theta\right]_{0 \leq j,k \leq L-1} \\
   & = \; \det\left[\frac{1}{\pi}\int_{-\pi}^\pi \cos\left(\left(j+\frac{1}{2}\right)\theta\right)
   \cos\left(\left(k+\frac{1}{2}\right)\theta\right)e^{2t\cos\theta}\d\theta\right]_{0 \leq j,k \leq L-1} \\
   & = \; \det\left[\frac{1}{\pi}\int_{-\pi}^\pi \frac{\cos\big(\big(j+\frac{1}{2}\big)\theta\big)}
   {\cos\frac{\theta}{2}}\frac{\cos\big(\big(k+\frac{1}{2}\big)\theta\big)}{\cos\frac{\theta}{2}}
   \cos^2\left(\frac{\theta}{2}\right) e^{2t\cos\theta}\d\theta\right]_{0 \leq j,k \leq L-1}.
\end{split}
\endeq
Setting $x=\cos\theta$ and noting that
$\frac{\cos((j+\frac{1}{2})\theta)}{\cos\frac{\theta}{2}}$ is a
constant multiple of a Jacobi polynomial
$P_j^{(-\frac{1}{2},\frac{1}{2})}(x)$ of the form
$2^jx^j+\cdots$, we have \eqarr \notag
D_L^{-+} & = & \frac{2^{L(L-1)}}{\pi^L}\det\left[\int_{-1}^1 x^{j+k}\sqrt{\frac{1+x}{1-x}}e^{2tx}\d x\right]_{0 \leq j,k \leq L-1} \\
   & = & \frac{2^{L(L-1)}}{\pi^L L!}\int_{[-1,1]^L}\prod_{1\leq k<\ell\leq L}|x_k-x_\ell|^2\prod_{j=1}^L\sqrt{\frac{1+x_j}{1-x_j}}e^{2tx_j}\d x_j.
\endeqarr
A steepest-descent analysis implies that \eq
\lim_{t\to\infty}D_L^{-+} \cdot \left(
\frac{e^{2tL}}{t^{L^2-L/2}\pi^L L!}\int_{[0,\infty]^L}\prod_{1\leq
k<\ell\leq L}|y_k-y_\ell|^2\prod_{j=1}^L y_j^{-1/2}e^{-y_j}\d y_j
\right)^{-1} = 1.
\endeq
The multiple integral is also a Selberg integral (see (17.6.5) of
\cite{Mehta:1991-book}), and we obtain\eq \label{DL-+}
\lim_{t\to\infty}D_L^{-+} \cdot \left(
\frac{e^{2tL}}{t^{L^2-L/2}\pi^L}\cdot\frac{G(L+1)G(L+\frac{1}{2})}{G(\frac{1}{2})}
\right)^{-1} = 1.
\endeq

Using~\eqref{DLm1++},~\eqref{DL-+}, and \eq \label{D2L-1}
\lim_{t\to\infty} D_{2L-1} \cdot \left(
\frac{e^{4tL-2t}}{(2t)^{2L^2-2L+\frac{1}{2}}(2\pi)^{L-\frac{1}{2}}}G(2L)
\right)^{-1} = 1,
\endeq
which follows from Section~\ref{sec-exact-part} (the exact part for
$F_2(x)$), we obtain \eq \label{D++D-+D} \lim_{t\to\infty}\left[
\log\frac{D_{L-1}^{++}D_L^{-+}}{D_{2L-1}} -
\log\left(\frac{2^{2L^2-L}}{\pi^{L-\frac{1}{2}}}\cdot\frac{[G(L+\frac{1}{2})]^2G(L)G(L+1)}{G(\frac{1}{2})G(\frac{3}{2})G(2L)}\right)
\right] = 0.
\endeq
The result is now proved using the properties~\eqref{G-of-one-half}
and~\eqref{G-at-large-z} for the Barnes G-function.
\end{proof}

Combining equation~\eqref{parts-of-2logE} and Lemmas
\ref{painleve-lemma}, \ref{airy-lemma}, and \ref{exact-lemma}, we
obtain \eq 2\log E(x) = \int_{-\infty}^x
\left(q(y)-\sqrt{\frac{-y}{2}}\right)\d y -
\frac{\sqrt{2}(-x)^{3/2}}{3} -\frac{1}{2}\log 2.
\endeq

\bibliographystyle{alpha}

\end{document}